\documentclass[11pt]{article}
\usepackage[section]{placeins}
\usepackage{multirow}
\usepackage{geometry}
\usepackage{amssymb}
\usepackage{amsmath}
\usepackage{float}
\usepackage{color}
\usepackage{graphicx}
\usepackage{subfigure}
\usepackage{tabu}
\usepackage{booktabs}
\usepackage{array}
\usepackage{caption}
\usepackage{amsthm}
\usepackage{booktabs}
\usepackage{cite}
\usepackage{indentfirst}
\numberwithin{equation}{section}
\newtheorem{theorem}{\bf Theorem}[section]
\usepackage{algpseudocode}
\usepackage{algorithm}
\usepackage{algorithmicx}
\newtheorem{lemma}{Lemma}[section]

\newtheorem{Corollary}{Corollary}[section]
\usepackage{booktabs}
\usepackage{hyperref}%, citecolor=green
\makeatletter
\@addtoreset{equation}{section}
\makeatother

\title{\bf An L-DEIM Induced High Order Tensor Interpolatory Decomposition}
\author{
	    Zhengbang Cao\footnote{School of Mathematical Sciences, Ocean University of China, Qingdao 266100, China.
		E-Mail: {\tt caozhengbang@stu.ouc.edu.cn}}
   \and Yimin Wei\footnote{School of Mathematical Sciences and and Key Laboratory of Mathematics for Nonlinear Sciences, Fudan University, Shanghai 200433, China.
		E-Mail: {\tt ymwei@fudan.edu.cn}}
\and	Pengpeng Xie\footnote{Corresponding author (P. Xie). School of Mathematical Sciences, Ocean University of China, Qingdao 266100, China.
		E-Mail: {\tt xie@ouc.edu.cn}.
	}
}

\date{}

\begin{document}
	\maketitle
	\begin{abstract}
		This paper derives the CUR-type factorization for tensors in the Tucker format based on a new variant of the discrete empirical interpolation method known as L-DEIM.
		This novel sampling technique allows us to construct an efficient algorithm for computing the structure-preserving decomposition, which significantly reduces the computational cost. %which needs much lower cost and provides a significant acceleration in practice.
%		For a given $d$-dimensional tensor $\mathcal{X}$, our algorithm provides the core tensor $\mathcal{G}$, along with a set of factors matrices ${\{C_n\}}_{n=1}^d$
%		which are columns extracted from the mode-$n$ tensor unfolding, preserving certain important features of the original tensor.
		For large-scale datasets, we incorporate the random sampling technique with the L-DEIM procedure to further improve efficiency. %that are expensive to store and manipulate, we combine the random sampling technique with the L-DEIM procedure to further improve the efficiency of our algorithm.
		Moreover, we propose randomized algorithms for computing a hybrid decomposition, which yield interpretable factorization and provide a smaller approximation error than the tensor CUR factorization.
		%Compared with the non-random algorithms, the proposed algorithms have advantages in speed with lower computational cost while keeping a high degree of accuracy.
%		We establish detailed probabilistic error analysis for the
%		algorithms and provide numerical results that show the promise of our approaches.
		We provide comprehensive analysis of probabilistic errors associated with our proposed algorithms, and present numerical results that demonstrate the effectiveness of our methods.
		\\ \hspace*{\fill} 
		\\ \hspace*{\fill} 
	\end{abstract}
% =========    Introduction
	\section{Introduction}
	\hskip 2em Tensor decompositions \cite{kolda2009tensorSIAM,kilmer2013SIAMthird,tucker1966PSYchometriKasome, Qi2021triple, carroll1970PSYanalysis, Delathauwer2000}, are efficient and widely used for multi-way data processing, and in particular, they can be utilized to compress the data tensors without destroying their intrinsic multidimensional structure.
	This work presents new algorithms for computing the CUR-type and hybrid CUR-type factorizations for tensors in the Tucker format based on a novel index selection procedure, that is, the L-discrete empirical interpolation method (L-DEIM) \cite{gidisu2022hybridArxiv}.
	Further, random sampling techniques \cite{martinsson2011ACHArandomized,halko2011SIAMfinding} are also utilized to enhance the efficiency of the proposed algorithms.
	
	\hskip 2em
	A CUR factorization \cite{sorensen2016deimSIAM,drineas2008SIAMrelative,hamm2021SIAMperturbations}
	is a low-rank approximation of a matrix $X \in \mathbb{R}^{m \times n}$ of the form
	\begin{equation}\label{matrix CUR}
		X\approx CUR,
	\end{equation}
	where $C$ and $R$ are the matrices that consist of actual columns and rows of $A$,
	inheriting certain important properties of the original matrix, such as sparsity,
	non-negativity, integer-values and so on.
	This novel property has rendered the CUR a potent tool for data analysis and attractive in a wide range of applications.
	%is a subset of the columns and rows of $X$ respectively,
%	and $k\times k$ matrix $U$ is constructed to ensure that matrix $CUR$ is a good approximation to $X$.
%	Compared to the best rank-$k$ approximation obtained by using the singular value decomposition (SVD), (\ref{matrix CUR}) is sub-optimal, but it preserves certain important features of the original matrix and is an appealing technique for identifying the underlying structure of a data matrix and extracting meaningful information in data analysis.	
	To deal with the multi-dimensional data, the CUR-type decomposition for tensors was proposed by
	\cite{drineas2007randomizedLAA,chen2022NFAOtensor,cortinovis2020SIAMlow}, and the perturbation analysis and sampling strategy were also researched in \cite{cai2021JMLRmode,ahmadi2021IEEEcross,che2022JOTAperturbations}.
	For tensors in the Tucker format, \cite{drineas2007randomizedLAA} provides a multilinear rank-$(r_1,r_2,\ldots,r_d)$ approximation for a given tensor $\mathcal{X}\in\mathbb{R}^{I_1\times I_2\times \cdots \times I_d}$ such that
	\begin{equation}\label{tensor CUR}
		\mathcal{X}\approx \mathcal{G} \times_{1} C_{1} \times_{2} C_{2} \times\cdots \times_{d} C_{d},
	\end{equation}
	where $\mathcal{G}\in\mathbb{R}^{r_1\times r_2\times \cdots\times r_d}$ is a core tensor, and the columns of matrices $\{C_n\}_{n=1}^d\in \mathbb{R}^{I_n\times r_n}$ are generated by sampling from the mode-$n$ fibers of $\mathcal{X}$, using a probability distribution that is dependent on the norms of the columns. %depending on a probability distribution based on the column norms.
	The author in \cite{saibaba2016hoidSIAM} derived a factorization of the form (\ref{tensor CUR}) based on the interpolatory decomposition \cite{cheng2005SIAMcompression,tyrtyshnikov2000COMPincomplete}, which was denoted as higher order interpolatory decomposition (HOID),
	where a number of sophisticated techniques for subset selection, including the DEIM \cite{barrault2004CRMempirical,chaturantabut2010deimSIAM},
	leverage score sampling \cite{jolliffe1972JRSSdiscarding,mahoneySIAM2008tensor}, strong rank-revealing QR (RRQR) \cite{gu1996SIAMefficient}
	and QR decomposition with column pivoting (PQR)\cite{drmac2016SIAMnew} were also extended to the domain of tensors.
	%related methods for subset selection, such as the DEIM \cite{barrault2004CRMempirical,chaturantabut2010deimSIAM}, leverage score sampling \cite{jolliffe1972JRSSdiscarding,mahoneySIAM2008tensor}, strong rank-revealing QR (RRQR) \cite{gu1996SIAMefficient} and QR decomposition with column pivoting (PQR)\cite{drmac2016SIAMnew} were extended to tensors.
	Numerical examples in \cite{saibaba2016hoidSIAM} demonstrate that the accuracy of the approximation heavily relies on the sampling technique utilized. %strongly depends on the sampling procedure, and the error incurred using the DEIM is comparable to the strong RRQR and much better than the leverage score approach.
	Results show that the DEIM method incurs errors that are comparable to those of the strong RRQR method, while outperforming the leverage score approach.
%	However, the notable drawback of the DEIM is that it requires the computation of the singular value decomposition (SVD) of each mode and the number of indices that can be selected is limited to the number of input singular vectors, making it hard for example in big data problems when computing the singular vectors of tensor unfolding is expensive.
	Nevertheless, it should be noted that the DEIM approach demands the computation of the singular value decomposition (SVD) for each mode, and the number of indices that can be chosen is limited by the number of input singular vectors.
	These characteristics make it challenging to apply DEIM in the context of big data problems, where the computation of singular vectors of tensor unfolding can prove to be a formidable task.%computing the singular vectors of tensor unfolding can prove to be a formidable task.%can be challenging.%big data problems, where computing the singular vectors of tensor unfolding is computationally expensive.
	
	\hskip 2em
	In recent times, a novel variant of the DEIM named L-DEIM has been introduced \cite{gidisu2022hybridArxiv}.
	This new method embodies a hybrid approach that leverages the advantageous properties of both deterministic leverage scores and DEIM,
	allowing for the selection of a larger number of indices than input singular vectors, while still achieving outcomes that are comparable to those of the original DEIM.
%	Recently, a new variant of the DEIM called L-DEIM, which is a combination of the strength of deterministic leverage scores and DEIM, was proposed in \cite{gidisu2022hybridArxiv}.
%	This method allows for the selection of a number of indices greater than the number of input singular vectors while enabling one to achieve comparable results to the original DEIM.
	In this paper, leveraging this novel sampling procedure, we develop efficient algorithms for
 	computing a multilinear rank-$(r_1,r_2,\ldots,r_d)$ approximation of the form (\ref{tensor CUR}).
	To be specific, during the process of constructing the approximation, the L-DEIM procedure operates on the $\widehat{r}_n$ right singular vectors of the mode-$n$ unfolding to select the indices, where $\widehat{r}_n\le r_n$, %and empirically, we can set $\widehat{r}_n=r_n/2$.
	and in practice, a value of $\widehat{r}_n=r_n/2$ has been found to yield favorable empirical results, just as shown in \cite{gidisu2022hybridArxiv}.
%	Consequently, our algorithm is particularly attractive for example in big data problems when computing ${r_n}$ singular value vectors of mode-$n$ unfolding is expensive even for moderately small $r_n$. %since it only requires a lower $\widehat{r}_n$ singular value vectors instead of $r_n$.
	Consequently, our proposed algorithm is especially advantageous in the scenarios involving large-scale data, where computing the $r_n$ singular value vectors of mode-$n$ unfolding is computationally expensive even for moderately small values of $r_n$.
	Despite its benefits, computationally, the L-DEIM induced HOID still necessitates the input of the SVD of each tensor unfolding, which can be prohibitively expensive for the tensors with large dimensions $d$ or for those with significant storage requirements. %in the case that the tensor $\mathcal{X}$ and the dimension $d$ are large.
%	Motivated by the success of the randomized algorithms researched in
%	\cite{che2021JSCefficient,che2019ACMrandomized,ahmadi2021IEEErandomized,che2020SIAMcomputation,che2020TCCTArandomized,che2021JCAMrandomized,minster2020SIAMrandomized},
%	 we adapt the random sampling techniques to improve the efficiency of our algorithm, which facilitate the procedure for matrix and tensor decompositions not only by reducing the computational complexity levels of deterministic algorithms but also reducing the communication among different levels of memories.
	Inspired by the remarkable achievements of randomized algorithms explored in
	\cite{che2021JSCefficient,che2019ACMrandomized,ahmadi2021IEEErandomized,che2020SIAMcomputation,che2020TCCTArandomized,che2021JCAMrandomized,minster2020SIAMrandomized},
	we incorporate random sampling techniques into our algorithm to enhance its efficiency,
	which facilitates matrix and tensor decompositions by not only decreasing the computational complexity of deterministic algorithms, but also reducing inter-level memory communication.
	Specifically, there are two distinct computational stages involving the processing of generating the factor matrices $\{C_n\}_{n=1}^d$.
%	The first stage is to make a reduction in all modes of the unfolding matrices with random sampling methods \cite{martinsson2011ACHArandomized} to construct a low-dimensional subspace that captures the actions of the unfolding matrices.
	In the first stage, we leverage random sampling methods \cite{martinsson2011ACHArandomized} to perform a mode-wise reduction of the unfolding matrices.
	This enables the construction of a low-dimensional subspace that faithfully captures the essential actions of the unfolding matrices.
%	Subsequently, the second stage can be completed with the L-DEIM procedure, operating on the approximated singular matrices obtained in in the first stage to sample the fibers.
	Subsequently, in the second stage, we implement the L-DEIM procedure to operate on the singular matrices approximated in the first stage.
	The aim of this procedure is to selectively sample the fibers, which constitute a pivotal component of the factorization process.
%	Moreover, in some applications, it is suggested to sample fibers only in some modes and not all of them.
%	In this situation, a hybrid CUR-type Tucker decomposition \cite{begovic2022hybridBIT} is preferred, whose approximation error is smaller than the factorization (\ref{tensor CUR}).
	In certain applications, it may be advisable to selectively sample fibers in only certain modes, rather than all of them.
	In such cases, a hybrid CUR-type Tucker decomposition, as proposed in \cite{begovic2022hybridBIT}, is preferred over the factorization (\ref{tensor CUR}), since it provides a smaller approximation error.
	By combining the randomized techniques and the sampling procedures such as the PQR, DEIM and L-DEIM, we provide three versions of efficient randomized algorithms for computing a hybrid CUR-type Tucker decomposition.
	Compared with the CUR-type and hybrid CUR-type Tucker decomposition algorithms based on the regular sampling procedure such as the PQR, RRQR and DEIM, our algorithms allow for a comparable accuracy with significantly lower cost and will be more computationally efficient on large-scale data.
	Details of the algorithm and theoretical analysis with the numerical results are provided to demonstrate the effectiveness of our approaches.
	
	\hskip 2em
	The rest of this article is organized as follows.
	In Section 2, we introduce some basic notation and describe several sampling techniques including the DEIM, the deterministic leverage score and the L-DEIM. Then we review some existing tensor decomposition, notably the higher order singular value decomposition (HOSVD)\cite{kolda2009tensorSIAM}, the HOID and the hybrid decomposition.
	Next, in Section 3 we present our algorithm for computing the HOID based on the Tucker factorization using the L-DEIM procedure, where the error bound is also presented in detail.
	In Section 4, we develop randomized algorithms for computing the HOID based on the sampling procedure L-DEIM, along with detailed probabilistic error analysis.
	The special cases where the dimensions of the input tensors are restricted to dimension 2, i.e., matrices are also considered.
	In Section 5, we provide new randomized algorithms for computing the hybrid decomposition and derive the probabilistic error analysis of the proposed algorithms.
	In Section 6, we test the performance of the proposed algorithms on several synthetic tensors and real-world datasets.
	Finally, in Section 7, we end this paper with concluding remarks.
	\section{Preliminaries}
	\subsection{Background on tensors}
	% ======Some background on tensors;
	\hskip 2em
	We begin by introducing fundamental notation and concepts for tensors.
	For a more comprehensive treatment, we refer readers to \cite{kolda2009tensorSIAM}.
%	Here we first review basic notation and concepts for tensors and more detailed properties can be found in \cite{kolda2009tensorSIAM}.
	A $d$-dimensional tensor is represented by $\mathcal{X}\in\mathbb{R}^{I_1\times I_2\times\cdots \times I_d}$ where the entries of $\mathcal{X}$ are denoted by%with entries given by
	\begin{equation*}
		x_{j_1 j_2 \ldots j_d},\quad 1\le j_1\le I_1, 1 \le j_2\le I_2,\ldots,1\le j_d\le I_d.
	\end{equation*}
	The norm of a $d$-dimensional tensor $\mathcal{X}$ with entries $x_{j_1 j_2 \ldots j_d}$ %is a generalization of the matrix Frobenius norm and
	is defined by
	\begin{equation*}
		\|\mathcal{X}\|_F
		=\sqrt{\sum_{j_1=1}^{I_1}\sum_{j_2=1}^{I_2}\cdots\sum_{j_d=1}^{I_d}|x_{j_1 j_2 \ldots j_d}|^2}.
	\end{equation*}
	The vector $2$-norm	and the matrix norm it induces are denoted by $\left\|\cdot\right\|$.
	$X_{(n)} \in \mathbb{R}^{I_n \times \prod_{k\neq n}I_k}$ represents the $n$th mode unfolding of the tensor $\mathcal{X}$.
%	and define
%	$I_n'=\mathrm{min}\{I_n,\prod_{k\neq n}I_k\}$.
	The multilinear rank of $\mathcal{X}$ is a tuple $(r_1,r_2,\ldots,r_d)$ where $r_n$ is the rank of $X_{(n)}$.
	The $n$-mode product of the tensor $\mathcal{X}$ with a matrix $U\in\mathbb{R}^{k\times I_n}$ is represented $\mathcal{X}\times_{n}U$, generating a tensor
	$\mathcal{Y}\in\mathbb{R}^{I_1\times \cdots \times I_{n-1}\times k \times I_{n+1}\times \cdots \times I_d}$, and elementwise,
%	such that
		\begin{equation*}
		\mathcal{Y}_{i_1\ldots i_{n-1} i i_{n+1}\ldots i_d}
		%=(\mathcal{X}\times_n U)_{i_1\ldots i_{n-1}ii_{n+1}\ldots i_d}
		=\sum\limits_{i_n=1}^{I_n}x_{i_1 \ldots i_d}u_{i i_n}.
	\end{equation*}
%	Alternatively it can be expressed in terms of matrix unfolding as
It can also be expressed in terms of matrix unfolding:
		\begin{equation*}
		\mathcal{Y}=\mathcal{X} \times_{n} {U} \quad \Leftrightarrow \quad
		{Y}_{(n)}={U} {X}_{(n)}.
		\end{equation*}
%	For distinct modes in a series of multiplications, it holds that
	For a series of multiplications involving distinct modes, the following relationships hold:
	\begin{equation}\label{property of mode product}
		\mathcal{X}\times_m Y \times_n Z=\mathcal{X}\times_n Z \times_m Y\quad (m \neq n),\quad
		\mathcal{X}\times_n Y \times_n Z=\mathcal{X}\times_n (ZY).
	\end{equation}
%	Throughout this paper, we use the MATLAB notation to index vectors and matrices, so that, e.g., $X(\mathbf{q},:)$ denotes the $k$ rows of $X$ whose indices are
%	specified by the entries of the vector $\mathbf{q}\in\mathbb{N}^k$, while $X(:,\mathbf{p})$ denotes the $k$ columns of $X$ indexed by $\mathbf{p}$.
	To index vectors and matrices, we use MATLAB notation throughout this paper.
	For example, $X(\mathbf{q},:)$ represents the $k$ rows of $X$ indexed by the indices set $\mathbf{q}\in\mathbb{N}^k$.
	
	%======= classical low=mulitirank approximation for tensors;
	\hskip 2em
	Now we are set to introduce several algorithms that produce a multilinear rank-$(r_1,r_2,\ldots,r_d)$ approximation to tensors in the Tucker format, i.e., the HOSVD, HOID and the hybrid algorithm.
	
	\hskip 2em
	Given a tensor $\mathcal{X}\in\mathbb{R}^{I_1\times I_2\times \cdots\times I_d}$, the HOSVD algorithm computes a core tensor $\mathcal{U}\in\mathbb{R}^{r_1\times r_2\times \cdots\times r_d}$ and a collection of matrices $U_j\in\mathbb{R}^{I_j \times r_j}$ containing the $r_j$ leading left singular vectors of $X_{(n)}$, $n=1,2,\ldots,d$ such that
	\begin{equation}\label{HOSVD}
		\mathcal{X}\approx \mathcal{U}\times_1 U_1 \times_2 U_2 \times\cdots \times_d U_d.
	\end{equation}
	Although the approximation error obtained by (\ref{HOSVD}) is theoretically smaller than that of (\ref{tensor CUR}), the CUR-type approximation facilitates interpreting the underlying data tensors and decomposing tensors so that their structures%, such as non-negativity, smoothness, or sparsity,
 can be potentially preserved.
	For certain applications where it may be appropriate to sample fibers in only some of the modes, rather than all of them, %it is not always important to keep the original entries in all modes, hence it is suggested to sample fibers only in some of the modes and not all of them.
	%Based on this idea, 
the author in \cite{begovic2022hybridBIT} proposed a hybrid CUR-type decomposition, which provides a decomposition such that
	\begin{equation}\label{hybrid decomposition}
		\mathcal{X}\approx\mathcal{G}\times_1 C_1\cdots\times_t C_t \times_{t+1}U_{t+1}\cdots\times_d U_d,
	\end{equation}
	where $\mathcal{G}\in\mathbb{R}^{r_1\times r_2\times \cdots\times r_d}$ is a core tensor, and the columns of matrices $\{C_i\}_{i=1}^t\in \mathbb{R}^{I_i\times r_i}$ are extracted from the mode-$i$ fibers of $\mathcal{X}$, while orthonormal matrices $\{U_j\}_{j=t+1}^d\in\mathbb{R}^{I_j\times r_j}$ are selected to minimize the approximation error.
	Consequently, the approximation error obtained this way is smaller than the one from (\ref{tensor CUR}).
	%
	%==============================Subset selection procedure
	%
	\subsection{Subset selection procedure}
	\hskip 2em
%	We now describe several approaches for subset selection that extract appropriate columns from $X_{(n)}$, that are the deterministic leverage score sampling procedure, the DEIM algorithm and the L-DEIM algorithm.
	We now give a concise introduction to several subset selection procedures.%, that are the deterministic leverage score sampling procedure, the DEIM algorithm and the L-DEIM algorithm.
	
	\hskip 2em
	Assume the best rank-$k$ SVD of
	$X\approx V\Sigma W^{\mathrm{T}}$ is available, where matrices $V$ and $W$ consist of the $k$ leading left and right singular vectors of $X$ respectively.
%	The rank-$k$ leverage score of the $i$th column of $X$ is defined as
%	\begin{equation*}
%		\ell_{i}=\left\|\left[W\right]_{i,:}\right\|^{2}, \quad i=1, 2, \ldots, n,
%	\end{equation*}
%	where $\left[W\right]_{i,:}$ denotes the $i$th row of $W$.
	The deterministic leverage score sampling algorithm extracts $k$ columns of $X$ corresponding to the indices of the largest leverage scores $\ell_{i}=\|W(i,:)\|^{2}.$ 
%where $W(i,:)$ denotes the $i$th row of $W$.
	While the simplicity of this approach has yielded remarkable success in practical applications, we note that a complete theoretical guarantee has yet to be established.% to the best of our knowledge.
	
%	From a practical perspective, this deterministic algorithm is extremely simple to implement, however, unfortunately, it does not admit provable performance guarantees.
	\hskip 2em
 	The DEIM \cite{chaturantabut2010deimSIAM} is an index selection procedure that gives simple, deterministic CUR factorizations for both matrices and tensors.
 %	It was first proposed in the context of model reduction of nonlinear dynamical systems.
 	The authors in \cite{sorensen2016deimSIAM} and \cite{saibaba2016hoidSIAM} utilized this procedure in the context of subset selection to CUR factorization for matrices and tensors respectively.
 %	To derive this method, we elaborate upon the interpolatory projectors.
 %	Given a set of distinct indices $\mathbf{p} \in \mathbb{N}^{k}$, the interpolatory projector for $\mathbf{p}$ onto the range of $V$ $\mathrm{Ran}(V)$ is defined as
 %	\begin{equation*}
% 		\mathcal{P}=V(P^\mathrm{T}V)^{-1}P^{\mathrm{T}},
% 	\end{equation*}
% 	where $P=I(:,\mathbf{p})$, provided $P^{\mathrm{T}}V$ is invertible.
% 	It can be readily verified that $\mathcal{P}$ is an oblique projector, and it has the following ``interpolatory property''.
% 	For any vector $x\in\mathbb{R}^{m}$, the projected vector $\mathcal{P}x$ matches $x$ in the $\mathbf{p}$ entries,
% 	\begin{equation*}
% 		(\mathcal{P}x)(\mathbf{p})
% 		=P^{\mathrm{T}}\mathcal{P}x
% 		=P^{\mathrm{T}}V(P^\mathrm{T}V)^{-1}P^{\mathrm{T}}x
% 		=P^{\mathrm{T}}x=x(\mathbf{p}).
% 	\end{equation*}
	Specifically, %the DEIM algorithm processes the columns of $V$ and $W$ sequentially starting with the first dominant singular vector.
	%Each step processes the next singular vector to produce the next index.
	%The selected indices are used to compute the interpolatory projector $\mathcal{P}$.
	%The next index is selected by removing the direction of the interpolatory projection in the previous vectors from the subsequent one and finding the index of the entry with the largest magnitude in the residual vector.
	the DEIM algorithm follows a sequential procedure to process the columns of matrices $V$ and $W$, beginning with the first dominant singular vector, and
	the next index corresponds to the largest magnitude in the residual vector.
	See the pseudocode in Algorithm \ref{Al-DEIM} for more details.
	
		\begin{algorithm}[htb]
		\caption{DEIM index selection \cite{chaturantabut2010deimSIAM}}
		\label{Al-DEIM}
		\hspace*{0.02in} {\bf Require:}
		$V \in \mathbb{R}^{m \times k}$, $W \in \mathbb{R}^{n \times k}$ with $k \leq \mathrm{min}(m,n)$.
		\begin{algorithmic}[1]
			\State  $v=V(:, 1)$.
			\State $p_{1}=\operatorname{argmax}_{1 \leq i \leq n}\left|v_{i}\right|$.
			\State $\mathbf{p}=[\begin{array}{l}
				p_1
			\end{array}]$.
			\For{$j=2,3, \ldots, k$}
				\State $v=V(:, j)$.
		%		\State $c=V(\mathbf{p}, 1: j-1)^{-1} v(\mathbf{p})$.
		%		\State $r=v-V(:, 1: j-1) {c}$.
				\State $r=v-V(:, 1: j-1) (V(\mathbf{p}, 1: j-1)\backslash v(\mathbf{p}))$.
				\State $p_{j}=\operatorname{argmax}_{1 \leq i \leq m}\left|{r}_{i}\right|$.
				\State $\mathbf{p}=\left[\begin{array}{ll}\mathbf{p} & p_{j}\end{array}\right]$.
			\EndFor
			\State Perform 1-9 on  $W$ to obtain the index $\mathbf{q}$.
			\State {\bf return} column and row index $\mathbf{q},\mathbf{p}\in\mathbb{N}^k_{+}$ respectively, with non-repeating entries.
		\end{algorithmic}
	\end{algorithm}

	\hskip 2em
%	In \cite{sorensen2016deimSIAM}, the DEIM algorithm was shown to be a viable index selection method for identifying the most representative and influential subset of columns and rows that define a low-dimensional space of the data.
%	However, a notable limitation of this index selection algorithm is that the number of indices that can be selected is limited to the number of available singular vectors.
	However, a major limitation of DEIM is its inability to select indices beyond the number of available singular vectors.
	%that it can only select indices up to the number of available singular vectors.
%	To surmount this limitation,
	To overcome this shortcoming, a novel variant of DEIM, called L-DEIM (Algorithm \ref{Al-LDEIM}), which integrates the advantage of leverage score sampling and the DEIM procedures, was proposed in \cite{gidisu2022hybridArxiv}.
	There are two principal steps involved in this method.
	Firstly, the standard DEIM procedure is executed to select the initial $\widehat{k}$ indices.
	%this method performs the standard DEIM procedure to select the first $\widehat{k}$ indices.
%	Then it computes the 2-norm of the rows of the residual singular vectors to select the additional $k-\widehat{k}$ indices.
	Subsequently, the 2-norm of the rows of the residual singular vectors is computed to select the additional $k-\widehat{k}$ indices.
	By adopting this method, only $\widehat{k}$ singular vectors are required to obtain $k$ indices.
	It is concluded in \cite{gidisu2022hybridArxiv} that L-DEIM is computationally more efficient than the original DEIM, while the accuracy of both methods can be comparable if the parameter $\widehat{k}$ is chosen appropriately.
	Besides, the L-DEIM algorithm degenerates to the DEIM algorithm, if we set $\widehat{k} = k$.
	\begin{algorithm}[htb]
		\caption{L-DEIM index selection \cite{gidisu2022hybridArxiv}}
		\label{Al-LDEIM}
		\hspace*{0.02in} {\bf Require:}
		$V \in \mathbb{R}^{m \times \widehat{k}}$ and $W \in \mathbb{R}^{n \times \widehat{k}}$, target rank $k$ with $\widehat{k} \leq k \leq \min (m, n)$.
		\begin{algorithmic}[1]
			\For{$j=1,2,\ldots,\widehat{k}$}
			\State $\mathbf{p}(j)=\operatorname{argmax}_{1 \leq i \leq m}\left|(V(i, j))\right|$.
			\State $V(:, j+1)=V(:, j+1)-V(:, 1: j) \cdot(V(\mathbf{p}, 1: j) \backslash V(\mathbf{p}, j+1))$.
			\EndFor
			\State Compute $\ell_{i}=\left\|V(i,:)\right\| \quad$ for $i=1,2, \ldots,m$.
			\State Sort $\ell$ in non-increasing order.
			\State Remove entries in $\ell$ corresponding to the indices in $\mathbf{p}$.
			\State $\mathbf{p}^{\prime}=k-\widehat{k}$ indices corresponding to $k-\widehat{k}$ largest entries of $\ell$.
			\State $\mathbf{p}=\left[\mathbf{p} ; \mathbf{p}^{\prime}\right]$.
			\State Perform 1-9 on $W$ to get index set $\mathbf{q}$.
			\State {\bf return} column and row indices $\mathbf{q},\mathbf{p}\in\mathbb{N}^k_{+}$ respectively, with non-repeating entries.
		\end{algorithmic}
	\end{algorithm}
	\section{L-DEIM based HOID}
	\hskip 2em
	This section derives a new variant of the HOID algorithm for representing low multilinear rank tensors $\mathcal{X}\in\mathbb{R}^{I_1\times I_2\times\cdots \times I_d}$ of the form (\ref{tensor CUR}) based on the L-DEIM procedure.
	
	\hskip 2em
	As mentioned earlier, the factor matrices $\{C_{n}\}_{n=1}^d$ of (\ref{tensor CUR}) are formed by extracting $r_n$ columns from the mode-$n$ tensor unfolding $X_{(n)}$, where the index sets of the selected columns are denoted by $\mathbf{s}$.
	We assume that the best rank-$\widehat{r}_n$ SVD $X_{(n)}\approx V_n\Sigma W_n^{\mathrm{T}}$, $\widehat{r}_n\le r_n$
	are available.
	Then we compute $\mathbf{s}$ by applying the L-DEIM algorithm on the matrix $W_{n}$.
	Once $\{C_n\}_{n=1}^d$ are obtained, the core tensor is computed as
	\begin{equation}\label{how core tensor}
		\mathcal{G}=
		\mathcal{X} \times_{1} {C}_{1}^{\dagger} \times_{2} {C}_{2}^{\dagger}\cdots \times_{d} {C}_{d}^{\dagger},
	\end{equation}
	which is optimal in the Frobenius norm \cite{sorensen2016deimSIAM}, where ${C}_{i}^{\dagger}$ is the Moore-Penrose inverse of ${C}_{i}$ \cite{Wei2018}.
	
	\hskip 2em
	We introduce our L-DEIM based HOID algorithm in Algorithm \ref{Al-LDEIM-HOID}.
	%, where the backslash operator is a Matlab-type notation for solving linear systems and least-squares.
	As stated in \cite{saibaba2016hoidSIAM}, for certain applications, it is often unnecessary to compute the core tensor $\mathcal{G}$.
	%in several applications, the additional step of computing the core tensor $\mathcal{G}$, by projecting the columns onto the original tensor $\mathcal{X}$, is not necessary and may be skipped.
	Utilizing the novel subset selection algorithm L-DEIM, our proposed algorithm allows to form a multilinear rank-$(r_1, r_2, \ldots,r_d)$ approximation for the given tensor $\mathcal{X}$ without having to compute $r_n$ right singular vectors of mode-$n$ tensor unfolding $X_{(n)}$, $n=1,2, \ldots,d$ and it only requires a smaller $\widehat{r}_n$ instead.
	Hence, the new method is particularly appealing in the cases where computing the singular vectors is computationally expensive and the dimension $d$ is prohibitively large.%attractive in a setting where computing the singular vectors is expensive and the dimension $d$ is too large.
	\begin{algorithm}[htb]
		\caption{L-DEIM based HOID}
		\label{Al-LDEIM-HOID}
		\hspace*{0.02in} {\bf Require:}
		$\mathcal{X}\in\mathbb{R}^{I_1\times I_2\times \cdots \times I_d}$, desired multilinear rank $(r_1,r_2,\ldots,r_d)$ and parameter $(\widehat{r}_1,\widehat{r}_2,\ldots,\widehat{r}_d)$.
		\begin{algorithmic}[1]
			\For{$n=1,2,\ldots,d$}
			\State Compute $\widehat{r}_n$ right singular vectors $W_{n}\in \mathbb{R}^{\prod_{k\neq n}I_k\times \widehat{r}_n}$ of mode-$n$ tensor unfolding $X_{(n)}$.
			\For{$j=1,2,\ldots,\widehat{r}_n$}
			\State $\mathbf{s}(j)=\operatorname{argmax}_{1 \leq i \leq \prod_{k\neq n}I_k}\left|(W_{n}(i, j))\right|$.
			\State $W_{n}(:, j+1)=W_n(:, j+1)-W_n(:, 1: j) \cdot(W_n(\mathbf{s}, 1: j) \backslash W_n(\mathbf{s}, j+1))$.
			\EndFor
			\State Compute $\ell_{i}=\left\|W_n(i,:)\right\| \quad$ for $i=1,2, \ldots, \prod_{k\neq n}I_k$.
			\State Sort $\ell$ in non-increasing order.
			\State Remove entries in $\ell$ corresponding to the indices in $\mathbf{s}$.
			\State $\mathbf{s}^{\prime}=r_n-\widehat{r}_n$ indices corresponding to $r_n-\widehat{r}_n$ largest entries of $\ell$.
			\State $\mathbf{s}=\left[\mathbf{s} ; \mathbf{s}^{\prime}\right]$.
			
			\State Form $C_n=X_{(n)}(:,\mathbf{s})$.
			\EndFor
			\State Compute core tensor
			$\mathcal{G}\in\mathbb{R}^{r_1\times r_2\times\cdots \times r_d}$ as
			%\begin{equation*}
			$	\mathcal{G}=\mathcal{X}\times_1 C^{ \dagger}_1\times_2 C^{\dagger}_2\cdots \times_d C^{\dagger}_d$.
			%\end{equation*}
			\State {\bf return} Tucker decomposition
			$\mathcal{X}\approx\mathcal{G}\times_1C_1 \times_2 C_2  \cdots\times_d C_d$.
		\end{algorithmic}
	\end{algorithm}

	\hskip 2em
	We now derive an estimate for the error incurred to produce a HOID based on the L-DEIM.
	Before that, we first present a result related to \cite[Lemma 2.1]{saibaba2016hoidSIAM}.%, and use it to derive our main result.

	\begin{lemma}\label{convert tensor norm to matrix}
		Let the factor matrices $C_n, n=1,2,\ldots,d$ and the core tensor $\mathcal{G}$ be computed as in (\ref{how core tensor}), then we have the following error bound
		\begin{equation*}
			\left\|\mathcal{X}- \mathcal{G} \times_{1} C_{1}  \times_{2} C_{2} \cdots \times_{d} C_{d} \right\|_F^2
			\le
			\sum_{n=1}^{d} \left\{\mathrm{min}( I_n,\prod_{k\neq n}I_k )      \left\| (I-C_nC_n^{\dagger})X_{(n)} \right\|^2\right\}.
		\end{equation*}
%	where $I_n'=\mathrm{min}\{I_n,\prod_{k\neq n}I_k\}$ as defined in the preliminaries.
	\end{lemma}
	\begin{proof}
		By the definition of core tensor $\mathcal{G}$ and the property (\ref{property of mode product}), we have
		\begin{equation*}
			\mathcal{G} \times_{1} C_{1} \times_{2} C_{2}  \cdots \times_{d} C_{d}
			=\mathcal{X} \times_{1} \left(C_{1} C_{1}^{\dagger}\right) \times_{2} \cdots \times_{d} \left(C_{d} C_{d}^{\dagger}\right).
		\end{equation*}
	Recalling the result in \cite[Lemma 2.1]{saibaba2016hoidSIAM}, for orthogonal projections $\{\Pi_i\}_{i=1}^n$, we have
	\begin{equation*}
		\| \mathcal{X}-\mathcal{X}\times_1\Pi_1  \times_2\Pi_2 \cdots \times_d\Pi_d  \|_F^2
		\le \sum_{n=1}^d \|\mathcal{X}-\mathcal{X}\times_n\Pi_n\|_F^2.
	\end{equation*}
%	for projections $\Pi_1,\Pi_2,\ldots,\Pi_d$,
	Observe that $C_nC_n^{\dagger}$ is an orthogonal  projection matrix, it follows that
	\begin{equation*}
		\begin{aligned}
			\|\mathcal{X}- \mathcal{G} \times_{1} C_{1} \times_{2} C_{2} \cdots \times_{d} C_{d}\|_F^2
			= &
			\|\mathcal{X}- \mathcal{X} \times_{1} \left(C_{1}C_1^{\dagger}\right)
			 \cdots
			\times_{d} \left(C_{d}C_d^{\dagger}\right)\|_F^2
			\\
			\le&
			\sum_{n=1}^{d} \|\mathcal{X}-\mathcal{X}\times_n (C_nC_n^{\dagger})\|_F^2
			\\
			\le&
			\sum_{n=1}^{d}\| (I-C_nC_n^{\dagger})X_{(n)} \|_F^2
			\\
			\le&
			\sum_{n=1}^{d} \left\{\mathrm{min}(I_n,\prod_{k\neq n}I_k) \left\| (I-C_nC_n^{\dagger})X_{(n)} \right\|^2\right\}.
		\end{aligned}
	\end{equation*}
	\end{proof}
 	The following theorem quantifies the error of the HOID produced	by Algorithm \ref{Al-LDEIM-HOID}.

 %=====Err bound of the L-DEIM based HOID
	\begin{theorem}
		Suppose that $\mathcal{X}\in\mathbb{R}^{I_1\times    I_2 \times \cdots\times I_d}$
		with $I_n\le\prod_{k\neq n}I_k$ for $n=1,2\ldots,d$.
		Let the matrices $C_n$ for $n = 1,2,\ldots,d$ and the core tensor $\mathcal{G}$ be
		computed according to Algorithm \ref{Al-LDEIM-HOID}.
		Then we have the following error bound
		\begin{equation*}
			\left\|\mathcal{X}-\mathcal{G} \times_{1} C_{1} \times_{2} C_{2} \cdots \times_{d} C_{d}\right\|_{F}^{2} \le
			\sum\limits_{n=1}^d
			\left[
				I_n
				\left(\prod_{k\neq n}I_k\right)			
					\frac{ 4^{\widehat{r}_n}\cdot \widehat{r}_n}{3}
			\right]
			\sigma_{\widehat{r}_n+1}^2,
		\end{equation*}
	where $\sigma_{\widehat{r}_{n}+1}$ is the $(\widehat{r}_n+1)$th largest singular value of the mode-$n$ tensor unfolding $X_{(n)}$.
	\end{theorem}
% PROOF　for the Err of L-DEIM HOID
	\begin{proof}
	It follows from Lemma \ref{convert tensor norm to matrix} that
		$$\left\|\mathcal{X}- \mathcal{G} \times_{1} C_{1} \times_{2} C_{2} \cdots \times_{d} C_{d} \right\|_F^2
		\le
		\sum\limits_{n=1}^{d} \left\{I_n \| (I-C_nC_n^{\dagger})X_{(n)} \|^2\right\}.$$
	Let $\mathbf{s}$ be the indices obtained by performing the L-DEIM to the right singular matrices $W_n$ of $X_{(n)}$ and set $S=I(:,\mathbf{s})$.
	Denote the interpolatory projectors $\mathbb{S}=S(W_{n}^{\mathrm{T}}S)^{\dagger}W_n$.
	Combining the result in \cite[Lemma 3]{hendryx2021extended} that
	$	\left\|(I-C_nC_n^{\dagger})X_{(n)}\right\| \le \left\| X_{(n)}(I-\mathbb{S}) \right\|$ with \cite[Lemma 2]{hendryx2021extended}, we have 
%the following inequality for the error in the interpolatory projections of $X_{(n)}$,
	\begin{equation*}
		\begin{aligned}
			\left\|(I-C_nC_n^{\dagger})X_{(n)}\right\|
			\le&
			\left\|X_{(n)}(I-\mathbb{S})\right\|
			=
			\left\|X_{(n)}(I-W_n W_n^{\mathrm{T}})(I-\mathbb{S})\right\|
			\\
			\le&
			\left\|X_{(n)}(I-W_n W_n^{\mathrm{T}})\right\|  \left\|I-\mathbb{S}\right\|
			=
			\left\|I-\mathbb{S}\right\| \sigma_{\widehat{r}_n+1}
           \\ 
           = &\|(W_n^{\mathrm{T}}S)^{\dagger}\|,
		\end{aligned}
	\end{equation*}
where we use the fact that $\left\|I-\mathbb{S}\right\|=\left\|\mathbb{S}\right\|=\|(W_n^{\mathrm{T}}S)^{\dagger}\|$ for $\mathbb{S}\neq 0$ or $I$ \cite{szyld2006many}.
%	Since $\mathbb{S}\neq 0$ or $I$, it is known that \cite{szyld2006many}
%	\begin{equation*}
%		\left\|I-\mathbb{S}\right\|=\left\|\mathbb{S}\right\|=\|(W_n^{\mathrm{T}}S)^{\dagger}\|.
%	\end{equation*}
	%as long as $\mathbb{S}\neq 0$ or $I$; see, e.g. \cite{szyld2006many}.
%	In \cite{gidisu2022rsvd}, the authors present an upper bound on the L-DEIM selection scheme error constant $\left\|(W_n^{\mathrm{T}}S)^{\dagger}\right\|$, that is
%%	\begin{equation*}
%	$	\left\|(W_n^{\mathrm{T}}S)^{\dagger}\right\|
%		\le
%		\sqrt{\frac{\widehat{r}_n\cdot\prod_{k\neq n}I_k}{3}}2^{\widehat{r}_n}$.
%	\end{equation*}
	%Finally, combining %the result in \cite[Lemma 3]{hendryx2021extended} that
%	\begin{equation*}
%	$	\left\|(I-C_nC_n^{\dagger})X_{(n)}\right\| \le \left\| X_{(n)}(I-\mathbb{S}) \right\|$
%	\end{equation*}
	%the above inequalities, 
Applying the result in \cite{gidisu2022rsvd} that $\left\|(W_n^{\mathrm{T}}S)^{\dagger}\right\|  \le  \sqrt{\frac{\widehat{r}_n\cdot\prod_{k\neq n}I_k}{3}}2^{\widehat{r}_n},$
we obtain the desired result.
\end{proof}
%========================== Section: Randomization for HOID========================
	\section{Randomization for HOID}
	\hskip 2em
	During the process of the DEIM and L-DEIM based HOID, the main cost lies in calculating the singular vectors of each mode tensor unfolding.
	%the computation for the  singular vectors of each mode tensor unfolding accounts for the main cost.
	When dealing with large-scale problems, the leading singular vectors can be effectively computed using %iterative methods \cite{lehoucq1998SAIMarpack,baglama2013BITimplicitly} or
	the randomized algorithms \cite{che2021JSCefficient,martinsson2011ACHArandomized}.
	Motivated by this success, in this section, utilizing the random sampling methods \cite{martinsson2011ACHArandomized}, we develop the randomized algorithms for computing the HOID based on the two sampling procedures.
	Moreover, we consider the scenario where the dimension of the input tensor is constrained to dimension 2, that is, matrices and provide a fast randomized algorithm for matrix CUR decomposition.
	%================== Subsection ===================================
	\subsection{Randomization for DEIM based HOID}
	\hskip 2em
%	It is well known that computing the full SVD of a matrix $X\in\mathbb{R}^{m\times n}$ costs $\mathcal{O}\left(n m^{2}\right)$, assuming $n \ge m$.
%	It is a well-established fact that
	The computation of the complete SVD of a matrix $X\in\mathbb{R}^{m\times n}$ costs $\mathcal{O}\left(n m^{2}\right)$, assuming $n \ge m$.
	However, this computational expense can be prohibitively high when the dimensions are large.
	The randomized SVD algorithm, developed in \cite{martinsson2011ACHArandomized}, provides a simple and efficient technique for generating an accurate approximation of the SVD for a given matrix, which comprise two distinct stages.
	
	\hskip 2em
%	In the first stage, we multiply the matrix $X$ by a Gaussian matrix whose entries are Gaussian random variables of
%	zero mean and unit variance to have random linear combinations of the rows of $X$, and then we form a matrix $Q$ that gives a good approximation of the range of $X$, which allows the approximation $X \approx XQQ^{\mathrm{T}}$.
	During the first stage, we multiply $X$ by a Gaussian random matrix with entries having zero mean and unit variance. %to multiply with the input matrix $X$.
	This results in a set of random linear combinations of the rows of $X$.
	Subsequently, we construct a matrix $Q$ that approximates the range of $X$, thereby yielding the approximation $X \approx XQQ^{\mathrm{T}}$.
	In the second stage, we compute a thin SVD of the much smaller matrix
	$XQ=V\Sigma U^\mathrm{T}$.
%	The resulting decomposition is truncated
	We then truncate the decomposition to the desired rank, and compute $W=QU$ to obtain the approximated right singular matrices.
	The operations yield the rank-$r$ approximation
	$X \approx XQQ^{\mathrm{T}} = V \Sigma W^{\mathrm{T}},$
	and the error bound
	\begin{equation}\label{Err bound for the randSVD}
		\left\|V \Sigma W^{\mathrm{T}}-X\right\| \leqslant
		\left(2 \sqrt{2 (r+p) m \beta^{2} \gamma^{2}+1}+2 \sqrt{2 (r+p) m} \beta \gamma\right) \sigma_{r+1}
	\end{equation}
	holds with probability not less than
	\begin{equation}\label{Pro for the RandSVD}
		\chi=1-
		\frac{1}{\sqrt{2 \pi(l-r+1)}}
		\left(\frac{e}{(l-r+1) \beta}\right)^{l-r+1}
		-
		\frac{1}{2\left(\gamma^{2}-1\right) \sqrt{\pi m \gamma^{2}}}
		\left(\frac{2 \gamma^{2}}{e^{\gamma^{2}-1}}\right)^{m},
	\end{equation}
	where $\sigma_{r+1}$ is the $(r+1)$th largest singular value of $X$, $l$ is a user-specified integer with $l = r+p$, and $p$ is the oversampling parameter utilized to augment the number of columns in order to enhance the flexibility of the computational method, $\beta$ and  $\gamma$ are positive real numbers such that $\gamma > 1$.
	To illustrate the use of these parameters, we choose $\beta = 3/4$, $\gamma^2 = 5$, and $p = 20$. With this choice, we can derive the above error bound, which holds with probability not less than $1 - 10^{-17}$.
	\cite[Table 1]{martinsson2011ACHArandomized}  presents similar results obtained by varying the values of $l$, $\beta$, and $\gamma$.
%	For example, choosing $\beta= 3/4$, $\gamma^2=5$ and $p = 20$, we obtain the above error bound with probability not less than $1-10^{-17}$, and \cite[Table 1]{martinsson2011ACHArandomized} contains the similar results obtained by taking
%	other values for $l$, $\beta$ and $\gamma$.
	
	\hskip 2em
	We summarize our randomized approach in Algorithm \ref{Al-R-DEIM-HOID}, where we obtain an approximation $\widehat{\mathcal{X}}$ of given tensor $\mathcal{X}$ in a CUR-type Tucker format
	\begin{equation*}
		\widehat{\mathcal{X}}
		=\mathcal{G} \times_{1} C_{1} \times_{2} C_{2} \cdots \times_{d} C_{d}.
	\end{equation*}
	In Algorithm \ref{Al-R-DEIM-HOID}, we leverage the randomization techniques in \cite{martinsson2011ACHArandomized} to expedite the SVD process and acquire the singular vectors of mode tensor unfolding $X_{(n)}$ for $n=1,2,\ldots,d$.
	Next, we exploit the DEIM index selection procedure, operating on the approximate singular vector matrices to identify the selection fibers and construct matrices $\{C_n\}_{n=1}^d$.
	Compared to the randomized approach proposed in \cite{saibaba2016hoidSIAM}, which performs well in numerical results, however, without a complete error analysis, we establish a detailed probabilistic error analysis for our randomized algorithm.

	%==========ALGORITHM: randomized DEIM-HOID
	\begin{algorithm}[htb]
		\caption{Randomized DEIM based HOID}
		\label{Al-R-DEIM-HOID}
		\hspace*{0.02in} {\bf Require:}
		$\mathcal{X}\in\mathbb{R}^{I_1\times I_2 \times \cdots \times I_d}$, multilinear rank $(r_1,r_2,\ldots,r_d)$ and oversampling parameter $p$.
		\begin{algorithmic}[1]
			\For{$n=1,2,\ldots,d$}
			\State{Draw random Gaussian matrix $\Omega\in\mathbb{R}^{(r_n+p)\times I_n}$.}
			\State{Compute $Y=\Omega X_{(n)}\in\mathbb{R}^{(r_n+p)\times \prod_{k\neq n}I_k}$.}
			\State Compute the SVD of $Y^{\mathrm{T}}$,
			%	\begin{equation*}
				$	Y^{\mathrm{T}}=ZMK^{\mathrm{T}}$,
			%	\end{equation*}
			\State where $Z\in\mathbb{R}^{\prod_{k\neq n}I_k \times (r_n+p)}$ and
			$W\in\mathbb{R}^{(r_n+p) \times (r_n+p)}$ are orthonormal, and
			$\Sigma\in\hspace*{0.20in}\mathbb{R}^{(r_n+p)\times (r_n+p)}$
			is diagonal.
			\State Form $Q=Z(:,1:r_n)$.
			\State Compute $T=X_{(n)}Q$.
			\State Compute the SVD of $T$,
			%	\begin{equation*}
				$	T=V\Sigma U^{\mathrm{T}}$,
			%	\end{equation*}
			\State where $V\in\mathbb{R}^{I_n \times r_n}$ and
				$U\in\mathbb{R}^{r_n \times r_n}$ are orthonormal, and
				$\Sigma\in\mathbb{R}^{r_n \times r_n}$ is diagonal.
			\State Compute $W_n=Q U\in\mathbb{R}^{\prod_{k\neq n}I_k \times r_n}$.
				\For{$j=1,2,\ldots,r_n$}
					\State
						$\mathbf{s}(j)=\operatorname{argmax}_{1 \leq i \leq \prod_{k\neq n}I_k}\left|\left(W_{n}(i, j)\right)\right|$.
					\State
						$W_{n}(:, j+1)=W_{n}(:, j+1)-W_{n}(:, 1: j) \cdot\left(W_{n}(\mathbf{s}, 1: j) \backslash W_{n}(\mathbf{s}, j+1)\right)$.
				\EndFor
			\State $C_n=X_n(:,\mathbf{s})$.
			\EndFor
			\State{Compute the core tensor
				$\mathcal{G}\in\mathbb{R}^{r_1\times r_2\times\cdots\times r_d}$ as
			%	\begin{equation*}
				$	\mathcal{G}=\mathcal{X}
					\times_1C_1^{\dagger} \times_{2} C_{2}^{\dagger} \cdots \times_{d} C_d^{\dagger}$.
			%	\end{equation*}
				\State {\bf return} Tucker decomposition
				$\widehat{\mathcal{X}}=\mathcal{G}\times_1C_1 \times_{2} C_{2}\cdots\times_dC_d$.
			}
		\end{algorithmic}
	\end{algorithm}
	% ERR for the randomized DEIM-HOID
	\begin{theorem}\label{ERR for the R-DEIM-HOID}
			Let $\mathcal{X}\in\mathbb{R}^{I_1\times I_2 \times \cdots \times I_d}$
			with $I_n\le\prod_{k\neq n}I_k$ for $n=1,2\ldots,d$, and
			$\widehat{\mathcal{X}}$ be an approximation for $\mathcal{X}$ provided by Algorithm \ref{Al-R-DEIM-HOID}.
			Suppose that $p$ is an oversampling parameter, $\beta$ and $\gamma$ are positive numbers such that $\gamma>1$, and
            $\phi=\prod_{n=1}^d \chi_n$ with
		\begin{equation*}
			\chi_n=1-\frac{1}{\sqrt{2\pi(p+1)}}\left(\frac{e}{(p+1)\beta}\right)^{p+1}
			-\frac{1}{2(\gamma^2-1)\sqrt{\pi I_n\gamma^2}}\left(\frac{2\gamma^2}{e^{\gamma^2-1}}\right)^{I_n}.
		\end{equation*}
%		$I=\mathtt{min}\{I_1,\ldots,I_d\}$.
		Then
	\begin{equation*}
		\begin{aligned}
			\|\mathcal{X}-\hat{\mathcal{X}}\|_F^2\le
			\sum\limits_{n=1}^d
			\Bigg[
			I_n
			\left(\prod_{k\neq n}I_k\right)
			\left(
			\frac{r_n \cdot 4^{r_n}}{3}
			\right)
			\left(
			2\sqrt{2(r_n+p)I_n\beta^2\gamma^2+1}
		%	\right.\\
		%	&\left.
			+2\sqrt{2(r_n+p)I_n}\beta\gamma
			\right)^2
			\Bigg]
			\sigma_{r_n+1}^2
		\end{aligned}
	\end{equation*}
		holds with probability not less than $\phi$, where $\sigma_{r_n+1}$ is the $(r_n+1)$th largest singular value of the mode-$n$ tensor unfolding $X_{(n)}$.
	\end{theorem}
	% Proof for the ERR of the randomized DEIM-HOID
	\begin{proof}
		First, from Lemma \ref{convert tensor norm to matrix}, we have
		\begin{equation}\label{Proof for R-DEIM HOID eq1}
			\left\|\mathcal{X}- \mathcal{G} \times_{1} C_{1} \times_{2} C_{2} \cdots \times_{d} C_{d} \right\|_F^2
			\le
			\sum_{n=1}^{d} \left(I_n \| (I-C_nC_n^{\dagger})X_{(n)} \|_2^2\right).
		\end{equation}
	 According to Algorithm \ref{Al-R-DEIM-HOID}, for $X_{(n)}$, $n=1,2,\ldots,d$, we have the approximate SVD
	 \begin{equation*}
	 	X_{(n)}=V_n \Sigma W_n^{\mathrm{T}} + E_n,
	 \end{equation*}
	 where $W_n$ contains $r_n$ approximated right singular vectors, and the error $\left\|E_n\right\|$ satisfies (\ref{Err bound for the randSVD}) with probability not less than (\ref{Pro for the RandSVD}).
	 Suppose that the column indices $\mathbf{s}$ give the full rank matrices $C=X_{(n)}S$ where $S=I(:,\mathbf{s})$, and let $\mathbb{S}=S(W_n^{\mathrm{T}}S)^{-1}W_n^{\mathrm{T}}$ be the interpolatory projectors. 
    Then, using the result in \cite[Lemma 4.2]{sorensen2016deimSIAM}, we have
	\begin{equation*}
		\left\| (I-CC^{\dagger})X_{(n)} \right\|
		\le
		\left\| X_{(n)}(I-\mathbb{S}) \right\|.
	\end{equation*}
	Note that $W_n^{\mathrm{T}}W_n=I$. According to \cite[Lemma 4.1]{sorensen2016deimSIAM}, we obtain that
	\begin{equation*}
		\left\|  X_{(n)}(I-\mathbb{S}) \right\|
		\le
		\left\| (W_n^{\mathrm{T}}S)^{-1} \right\|
		\left\| X_{(n)}  (I-W_n W_n^{\mathrm{T}})\right\|.
	\end{equation*}
	Then it follows that
	\begin{equation*}
		\begin{aligned}
			\left\| (I-CC^{\dagger})X_{(n)} \right\|
			\le&
			\left\| (W_n^{\mathrm{T}}S)^{-1} \right\|
			\left\| X_{(n)} (I-W_n W_n^{\mathrm{T}})  \right\|
			\\ =&
			\left\| (W_n^{\mathrm{T}}S)^{-1} \right\|
			\left\| (E_n+V_n\Sigma W_n^{\mathrm{T}}) (I-W_n W_n^{\mathrm{T}}) \right\|
			\\ \le&
			\left\| (W_n^{\mathrm{T}}S)^{-1} \right\|
			\left\| E_n (I-W_nW_n^{\mathrm{T}}) \right\|
			\le
			\left\| (W_n^{\mathrm{T}}S)^{-1} \right\|
			\left\| E_n \right\|.
		\end{aligned}
	\end{equation*}
	For the DEIM, it is shown in \cite[Lemma 4.4]{sorensen2016deimSIAM} that
	$\left\| (W_n^{\mathrm{T}}S)^{-1} \right\| \le \sqrt{\frac{r_n \prod_{k\neq n}I_k}{3}}2^{r_n}$,
	which implies that
	\begin{equation}\label{Proof for R-DEIM HOID eq2}
		\left\| (I-CC^{\dagger})X_{(n)} \right\|
		\le
		\sqrt{\frac{r_n \prod_{k\neq n}I_k}{3}}2^{r_n}
		\left(2 \sqrt{2(r_n+p) I_n \beta^{2} \gamma^{2}+1}+2 \sqrt{2(r_n+p)I_n} \beta \gamma\right) \sigma_{r_n+1}
	\end{equation}
	for $n=1,2,\ldots,d$ with probability not less than
	\begin{equation*}\label{Proof for R-DEIM HOID eq3}
		\chi_n=1
		-\frac{1}{\sqrt{2 \pi(p+1)}} \left(\frac{e}{(p+1) \beta}\right)^{p+1}
		-\frac{1}{2\left(\gamma^{2}-1\right) \sqrt{\pi I_n \gamma^{2}}}
		\left(\frac{2 \gamma^{2}}{e^{\gamma^{2}-1}}\right)^{I_n}.
	\end{equation*}
	Setting $\phi=\prod_{n=1}^d \chi_n$ and inserting relation
	 (\ref{Proof for R-DEIM HOID eq2}) into (\ref{Proof for R-DEIM HOID eq1}), we obtain the the desired error bound.
	\end{proof}

	\hskip 2em
	As pointed out in \cite{saibaba2016hoidSIAM}, given matrix $X\in\mathbb{R}^{m \times n}$, the relationship between the HOID and the matrix CUR factorization can be established effortlessly by recognizing the subsequent identity:%the connection of the HOID with the matrix CUR decomposition is readily established by noting the following identity
	\begin{equation}\label{relation of CUR and HOID eq1}
		X=CUR+E \iff X=U\times_{1} C \times_{2} R+E,
	\end{equation}
%	where $C\in\mathbb{R}^{m \times r}$ is a matrix which represents columns of $X$, while $R\in\mathbb{R}^{r \times n}$ contains sampled rows from $X$.
	where matrices $C\in\mathbb{R}^{m \times r}$ and $R\in\mathbb{R}^{r \times n}$ are formed by extracting the rows/columns of $X$.
%	With this choice of the intersection matrix, the connection with the core tensor computation follows from
	By adopting this specific intersection matrix, the correlation with the core tensor calculation can be derived:
	\begin{equation}\label{relation of CUR and HOID eq2}
		U=C^{\dagger} X R^{\dagger} \iff U=X \times_1 C^{\dagger} \times_2 R^{\dagger}.
	\end{equation}
	From relations (\ref{relation of CUR and HOID eq1}) and (\ref{relation of CUR and HOID eq2}), it becomes evident that Algorithm \ref{Al-R-DEIM-HOID} can also be applied to produce a matrix CUR decomposition and we summarize the error bound in the following corollary.
	\begin{Corollary}
		Apply Algorithm \ref{Al-R-DEIM-HOID} to produce the CUR decomposition of $X\in\mathbb{R}^{m\times n}$, $n\ge m$ as in (\ref{relation of CUR and HOID eq1}). Then
		\begin{equation*}
			\left\| E \right\|
			\le
			\left(\sqrt{\frac{mr}{3}}2^r+\sqrt{\frac{nr}{3}}2^r\right)
			\left(2 \sqrt{2(r+p) m \beta^{2} \gamma^{2}+1}+2 \sqrt{2(r+p)m} \beta \gamma\right)
			\sigma_{r+1}
		\end{equation*}
	with success probability not less than $\chi^2$.
	\end{Corollary}
%=====================subsection{Randomization for L-DEIM based matrix CUR decomposition}
	\subsection{Randomization for L-DEIM based matrix CUR decomposition}
	\hskip 2em 	
	We now focus on the integration of random sampling techniques with the L-DEIM algorithm, a combination that can yield good bounds with high probability at a trivial computational cost.
	To develop a framework for our randomized approaches, firstly, we consider the matrix case
%the special case when the dimensions of the input tensor are restricted to have dimension $2$,
 and derive a randomized algorithm for the matrix CUR decomposition of the form (\ref{matrix CUR}).
	
	\hskip 2em
	Suppose the selected indices are stored in the vectors $\mathbf{q},\mathbf{p}\in\mathbb{N}^{r}$ so that
	$C = X(:,\mathbf{q})$ and $R = X(\mathbf{p}, :)$.
	Our choice for $\mathbf{p}$ and $\mathbf{q}$ is guided by information of the approximate rank-$\widehat{r}$ SVD of $X\in\mathbb{R}^{m\times n}$ such that
	\begin{equation}\label{RSVD for LDEIM M-CUR}
		X\approx V\Sigma W^{\mathrm{T}},
	\end{equation}
	where matrices $W,V$ contain the leading $\widehat{r}$ right and left singular vectors and $\widehat{r}\le r$ is the user-specified parameter contained in the L-DEIM algorithm.
	Furthermore, we compute decomposition (\ref{RSVD for LDEIM M-CUR}) by applying the randomized technique, achieving the error (\ref{Err bound for the randSVD}) with probability not less than (\ref{Pro for the RandSVD}).
	Then we compute
	\begin{equation}\label{Def of U}
		U=C^{\dagger}XR^{\dagger},
	\end{equation}
	yielding a CUR factorization by two steps:
	first, the columns of $X$ are projected onto the range of $C$ $\mathrm{Ran}(C)$; then the result is projected onto the row space of $R$.
	This option minimizes $\|X-CUR\|$ for the given the sampling indices.
%	\begin{equation*}
%		M=CC^{\dagger}A, \quad CUR=MRR^{\dagger},
%	\end{equation*}
%	where both steps are optimal with respect to the two-norm error.

	\hskip 2em
	Lines $1$ to $9$ of Algorithm \ref{Al-R-LDEIM-CUR} correspond to the construction of rank-$\widehat{r}$ truncated SVD of $X$.
	Additionally, in line $2$, we multiply the matrix $X$ by an $(\widehat{r}+p)\times n$ Gaussian matrix $\Omega$ to implement truncation, and it would increase to  $(r+p)\times n$ if we apply the DEIM, which can be easily observed from line 3 of Algorithm \ref{Al-R-DEIM-HOID}.
	Therefore, by exploiting the L-DEIM technique, the random sampling procedure can be executed very efficiently by achieving a better truncation, which is the primary
	source of the excellent performance of our approach.
%	Besides, we note that we can parallelize the work in lines 10 to 19 since it consists of three independent runs of L-DEIM, which operate on the singular vector matrices $W$ and $V$ to select the row indices $\mathbf{q}$ and column indices $\mathbf{p}$ respectively.
	Besides, it is worth noting that lines 10 to 19 can be parallelized, as it involves three independent runs of L-DEIM, which operate on the singular vector matrices $W$ and $V$ to select the row indices $\mathbf{q}$ and column indices $\mathbf{p}$ respectively.
	The following theorem quantifies the error of the rank-$r$ CUR decomposition produced by Algorithm \ref{Al-R-LDEIM-CUR}.
	
	%------------ALGORITHM: Randomized L-DEIM based CUR decomposition
	\begin{algorithm}[htb]
		\caption{Randomized L-DEIM based CUR decomposition}
		\label{Al-R-LDEIM-CUR}
		\hspace*{0.02in} {\bf Require:}
		$X\in\mathbb{R}^{m\times n}$, desired rank $r$ and the specified parameter $\widehat{r}$.
		\begin{algorithmic}[1]
			% Process of randomization
		\State Draw random Gaussian matrix $\Omega\in\mathbb{R}^{(\widehat{r}+p)\times n}$.
		\State Compute $Y=\Omega X\in\mathbb{R}^{(\widehat{r}+p)\times n}$.
		\State Compute the SVD of $Y^{\mathrm{T}}$,
	%	\begin{equation*}
		$	Y^{\mathrm{T}}=ZMK^{\mathrm{T}}$,
	%	\end{equation*}
		\State where $Z\in\mathbb{R}^{n \times (\widehat{r}+p)}$ and
		$K\in\mathbb{R}^{(\widehat{r}+p) \times (\widehat{r}+p)}$ are orthonormal, and
		$M\in\mathbb{R}^{(\widehat{r}+p)\times (\widehat{r}+p)}$
		is diagonal.
		\State Form $Q=Z(:,1:\widehat{r})$.
		\State Compute $T=XQ \in\mathbb{R}^{m \times \widehat{r}}$.
		\State Compute the SVD of $T$,
		%\begin{equation*}
			$T=V\Sigma U^{\mathrm{T}}$,
		%\end{equation*}
		\State where $V\in\mathbb{R}^{m \times \widehat{r}}$ and
		$U\in\mathbb{R}^{\widehat{r} \times \widehat{r}}$ are orthonormal, and
		$\Sigma$$\in\mathbb{R}^{\widehat{r} \times \widehat{r}}$ is diagonal.
		\State Compute $W=Q U\in\mathbb{R}^{n \times \widehat{r}}$.
		\For{$j=1,2,\ldots,\widehat{r}$}
			\State $\mathbf{p}(j)=\operatorname{argmax}_{1 \leq i \leq m}\left|(V(i, j))\right|$.
			\State $V(:, j+1)=V(:, j+1)-V(:, 1: j) \cdot(V(\mathbf{p}, 1: j) \backslash V(\mathbf{p}, j+1))$.
		\EndFor
		\State Compute $\ell_{i}=\left\|V(i,:)\right\| \quad$ for $i=1, 2,\ldots, m$.
		\State Sort $\ell$ in non-increasing order.
		\State Remove entries in $\ell$ corresponding to the indices in $\mathbf{p}$.
		\State $\mathbf{p}^{\prime}=r-\widehat{r}$ indices corresponding to $r-\widehat{r}$ largest entries of $\ell$.
		\State $\mathbf{p}=\left[\mathbf{p} ; \mathbf{p}^{\prime}\right]$.
		\State Repeat step 10-18 for $W$ to obtain index set $\mathbf{q}$.
		\State Form $C=X(:,\mathbf{q})$ and $R=X(\mathbf{p},:)$.
		\State Compute $U=C^{\dagger}XR^{\dagger}$.
		\State {\bf return} CUR decomposition
		$X \approx CUR$.
		\end{algorithmic}
	\end{algorithm}

%========= Err bound of Randomized L-DEIM based matrix CUR.
		\begin{theorem}\label{Err of R-LDEIM CUR}
		Let $X\in\mathbb{R}^{m\times n}$ with $n\ge m$. Suppose that $p$ is an oversampling parameter, $\beta$ and $\gamma$ are positive numbers such that $\gamma > 1 $, and
		\begin{equation*}
			\chi=1-\frac{1}{\sqrt{2\pi(p+1)}}\left(\frac{e}{(p+1)\beta}\right)^{p+1}
				-\frac{1}{2(\gamma^2-1)\sqrt{\pi m\gamma^2}}\left(\frac{2\gamma^2}{e^{\gamma^2-1}}\right)^{m}.
		\end{equation*}
	Then

		\begin{equation*}
			\left\|X-CUR\right\|\le
			2^{\widehat{r}}
			\left(\sqrt{\frac{n\widehat{r}}{3}}+\sqrt{\frac{m\widehat{r}}{3}}\right)
			\left(
			2\sqrt{2\left(\widehat{r}+p\right)m\beta^2\gamma^2+1}
			+2\sqrt{2\left(\widehat{r}+p\right)m}\beta\gamma
			\right)\sigma_{\widehat{r}+1}
		\end{equation*}
	holds with probability not less than $\chi^2$, where $\sigma_{\widehat{r}+1}$ is the $(\widehat{r}+1)$th largest singular value of $X$.
	\end{theorem}
%    PROOF　for the L-DEIM based CUR
	\begin{proof}
	This proof is a minor modification of that of \cite[Lemma 4.2]{sorensen2016deimSIAM}.
	Here we closely follow their proof technique.
	From the definition of $U$ of (\ref{Def of U}),
	\begin{equation*}
	X-CUR=X-CC^{\dagger}XR^{\dagger}R=(I-CC^{\dagger})A+CC^{\dagger}X(I-R^{\dagger}R).
	\end{equation*}
Then we have
\begin{equation}\label{Proof for R-L-DEIM CUR eq1}
	\begin{aligned}
		\left\|X-CUR\right\|
		\le&
		\left\|(I-CC^{\dagger})X\right\|+ \left\|CC^{\dagger}\right\| \left\|X(I-R^{\dagger}R)\right\|
		\\
		=&
		\left\|(I-CC^{\dagger})X\right\|+ \left\|X(I-R^{\dagger}R)\right\|,
	\end{aligned}
\end{equation}
since $\left\|CC^{\dagger}\right\|=1$.
Let $P=I(:,\mathbf{p})$, $Q=I(:,\mathbf{q})$, and $\mathbb{P}=V(P^{\mathrm{T}}V)^{\dagger}P^{\mathrm{T}}$, $\mathbb{Q}=Q(W^\mathrm{T}Q)^{\dagger}W^{\mathrm{T}}$ be the interpolatory projectors.
 Using the formula $C=X(:,\mathbf{q})=XQ$, we have
\begin{equation*}
	C^{\dagger}=(C^{\mathrm{T}}C)^{-1}C^{\mathrm{T}}
	=(Q^{\mathrm{T}}X^{\mathrm{T}}XQ)^{-1}(XQ)^{\mathrm{T}},
\end{equation*}
and then the orthogonal projection of $X$ onto $\mathrm{Ran}(C)$ is
\begin{equation*}
	CC^{\dagger}X%(XQ(Q^{\mathrm{T}}X^{\mathrm{T}}XQ)^{-1}Q^{\mathrm{T}}X^{\mathrm{T}})X
	=XQQ^{\mathrm{T}}X^{\mathrm{T}}XQ^{-1}Q^{\mathrm{T}}X^{\mathrm{T}}X.
\end{equation*}
Hence the error in the orthogonal projection of $X$ is
\begin{equation*}
	(I-CC^{\dagger})X=X(I-\Phi),\quad
	\Phi=Q(Q^{\mathrm{T}}X^{\mathrm{T}}XQ)^{-1}Q^{\mathrm{T}}X^{\mathrm{T}}X.
\end{equation*}
It is easy to verify that $\Phi Q=Q$.
Therefore, we obtain
\begin{equation*}
	\Phi\mathbb{Q}=\Phi Q(W^{\mathrm{T}}W)^{\dagger}W^{\mathrm{T}}=Q(W^{\mathrm{T}}Q)^{\dagger}W^{\mathrm{T}}
	=\mathbb{Q},
\end{equation*}
which implies that
\begin{equation*}
	X(I-\Phi)=X(I-\Phi)(I-\mathbb{Q})=(I-CC^{\dagger})X(I-\mathbb{Q}).
\end{equation*}
Then it follows that
\begin{equation}\label{Proof for R-L-DEIM CUR eq2}
	\begin{aligned}
		\left\| (I-CC^{\dagger})X \right\|
		=&
		\left\| X(I-\Phi) \right\|
		\\
		=&
		\left\| (I-CC^{\dagger})X(I-\mathbb{Q}) \right\|
		\\
		\le&
		\left\|I-CC^{\dagger}\right\|  \left\|X(I-\mathbb{Q})\right\|
		=
		\left\|X(I-\mathbb{Q})\right\|.
	\end{aligned}
\end{equation}
Analogous manipulation gives
\begin{equation}\label{Proof for R-L-DEIM CUR eq3}
	\left\|X(I-R^{\dagger}R)\right\| \le \left\|(I-\mathbb{P})X\right\|.
\end{equation}
Note that oblique projectors $\mathbb{P}$ and $\mathbb{Q}$ have the properties
$\mathbb{P}V=V$ and $W^{\mathrm{T}}\mathbb{Q}=W^{\mathrm{T}}$, so that
$(I-\mathbb{P})V=0$ and $W^{\mathrm{T}}(I-\mathbb{Q})=0$.
Therefore,
\begin{equation}\label{Proof for R-L-DEIM CUR eq4}
	\left\|X-\mathbb{P}X\right\|
	=\left\|(I-\mathbb{P})X\right\|
	=\left\|(I-\mathbb{P})(I-VV^{\mathrm{T}})X\right\|
	\le
	\left\|(I-VV^{\mathrm{T}})X\right\| \left\|(I-\mathbb{P})\right\|,
\end{equation}
\begin{equation}\label{Proof for R-L-DEIM CUR eq5}
	\left\|X-X\mathbb{Q}\right\|
	=\left\|X(I-\mathbb{Q})\right\|
	=\left\|X(I-WW^{\mathrm{T}})(I-\mathbb{Q})\right\|
	\le
	\left\|X(I-WW^{\mathrm{T}})\right\| \left\|(I-\mathbb{Q})\right\|.
\end{equation}
According to the description of the randomized SVD, the error $E$ between matrix $X\in \mathbb{R}^{m\times n}$ and its approximation satisfies the following inequality
\begin{equation*}
	\left\|E\right\|=\left\|X-XQQ^{\mathrm{T}}\right\|\le
	\left(2 \sqrt{2 (\widehat{r}+p)m \beta^{2} \gamma^{2}+1}+2 \sqrt{2 (\widehat{r}+p) m } \beta \gamma\right) \sigma_{\widehat{r}+1}
\end{equation*}
with probability not less than $\chi$ as defined in (\ref{Err bound for the randSVD}).
Therefore,
\begin{equation}\label{Proof for R-L-DEIM CUR eq6}
	\begin{aligned}
		\left\|(I-VV^{\mathrm{T}})X\right\|
		=&
		\left\|(I-VV^{\mathrm{T}})(XQQ^{\mathrm{T}}+E)\right\|
		\\
		\le &
		\left\|(I-VV^{\mathrm{T}})XQQ^{\mathrm{T}}+(I-VV^{\mathrm{T}})E\right\|
		\\
		\le&
		\left\|(I-VV^{\mathrm{T}})XQQ^{\mathrm{T}}\right\|+\left\|E\right\|
		= \left\|E\right\|,
	\end{aligned}
\end{equation}
since $(I-VV^{\mathrm{T}})XQQ^{\mathrm{T}}=(I-VV^{\mathrm{T}})V\Sigma W^{\mathrm{T}}=0$.
 A similar treatment shows that
\begin{equation}\label{Proof for R-L-DEIM CUR eq7}
	\left\|X(I-WW^{\mathrm{T}})\right\|\le \left\|E\right\|
\end{equation}
with probability not less than $\chi$.
Finally, combining the results from \cite{szyld2006many} and \cite{gidisu2022rsvd} that
\begin{equation*}
	\left\|I-\mathbb{P}\right\|
	=\left\|\mathbb{P}\right\|
	=\left\|(P^{\mathrm{T}}V)^{\dagger}\right\|
	\le \sqrt{\frac{m \widehat{r}}{3}} 2^{\widehat{r}},
	\\
	\left\|I-\mathbb{Q}\right\|
	=\left\|\mathbb{Q}\right\|
	=\left\|(W^{\mathrm{T}}Q)^{\dagger}\right\|
	\le \sqrt{\frac{n \widehat{r}}{3}} 2^{\widehat{r}},
\end{equation*}
and the relations (\ref{Proof for R-L-DEIM CUR eq1})-(\ref{Proof for R-L-DEIM CUR eq7}), we obtain the desired error bound.
\end{proof}

%==============================Subsection ==========================================
	\subsection{Randomization for L-DEIM based HOID}
 %\hskip 2em
	In this subsection, we design an efficient randomized algorithm for computing a CUR-type factorization for tensors in the Tucker format based on the L-DEIM procedure, which can be viewed as a generalization of Algorithm \ref{Al-R-LDEIM-CUR}.
	In this circumstance, each mode of tensor $\mathcal{X}\in \mathbb{R}^{I_1\times I_2\times\cdots \times I_d}$ is processed separately.
	Specifically, the factor matrices $C_n=X_{(n)}(:,\mathbf{s})$ are constructed by extracting $r_n$ columns from the $n$-mode unfolding $X_{(n)}$,
	where $\mathbf{s}$ represents the index sets of the selected columns.
	The selection of $\mathbf{s}$ is achieved by applying the L-DEIM algorithm to the approximate right singular matrices $W_n$,
	computed by employing the random sampling method as described in Section 4.1.
	Once all factor matrices ${\{C_n\}}_{n=1}^d$ are obtained, the core tensor is formed as
	\begin{equation*}
		\mathcal{G}=\mathcal{X}\times_1 C_1^{\dagger}\times_2 C_2^{\dagger}\cdots \times_d C_d^{\dagger}.
	\end{equation*}
%	This choice of the core tensor is optimal in the Frobenius norm and further details are provided in \cite{saibaba2016hoidSIAM}.
	Algorithm \ref{Al-R-LDEIM-HOID} is a summary of this procedure and it has several advantages:
	(1) it returns a HOID factorization that is known to be more interpretable than the HOSVD as it corresponds to representing data via other actual data points;
	(2) it has a computational advantage: %the L-DEIM induced randomized algorithm boasts one advantage that Algorithm \ref{Al-R-DEIM-HOID} does not have:
	the main cost of  the random truncation process presented in lines 2 to 10 of Algorithm \ref{Al-R-DEIM-HOID} comes from the SVD computation and it needs
	$\mathcal{O}\left((r_i+p)^2\prod_{k\neq i}I_k+I_i{r_i}^2\right)$. Since it requires fewer singular vectors in Algorithm \ref{Al-R-LDEIM-HOID}, it reduces to
	$\mathcal{O}\left((\widehat{r}_i+p)^2\prod_{k\neq i}I_k+I_i{\widehat{r}_i}^2\right)$.
	(3) there is a good theoretical guarantee for its performance, and we establish it in the following theorem.
	
\begin{algorithm}[htb]
	\caption{Randomized L-DEIM based HOID}
	\label{Al-R-LDEIM-HOID}
	\hspace*{0.02in} {\bf Require:}
	$\mathcal{X}\in\mathbb{R}^{I_1\times I_2\times\cdots \times I_d}$, multilinear rank $(r_1,r_2,\ldots,r_d)$ and parameters $(\widehat{r}_1,\widehat{r}_2,\ldots,\widehat{r}_d)$.
	\begin{algorithmic}[1]
		\For{$n=1,2,\ldots,d$}
		\State Draw random Gaussian matrix $\Omega\in\mathbb{R}^{(\widehat{r}_n+p)\times I_n}$.
		\State Compute $Y=\Omega X_{(n)}\in\mathbb{R}^{(\widehat{r}_n+p)\times \prod_{k\neq n}I_n}$.
		\State Compute the SVD of $Y^{\mathrm{T}}$,
	%	\begin{equation*}
		$	Y^{\mathrm{T}}=ZMK^{\mathrm{T}}$,
	%	\end{equation*}
		\State where $Z\in\mathbb{R}^{\prod_{k\neq n}I_k \times (\widehat{r}_n+p)}$ and
		$W\in\mathbb{R}^{(\widehat{r}_n+p) \times (\widehat{r}_n+p)}$ are orthonormal, and
		$\Sigma\in\hspace*{0.20in}\mathbb{R}^{(\widehat{r}_n+p)\times (\widehat{r}_n+p)}$
		is diagonal.
		\State Form $Q=Z(:,1:\widehat{r}_n)$.
		\State Compute $T=X_{(n)}Q$.
		\State Compute the SVD of $T$,
	%	\begin{equation*}
		$	T=V\Sigma U^{\mathrm{T}}$,
	%	\end{equation*}
		\State where $V\in\mathbb{R}^{I_n \times \widehat{r}_n}$ and
		$U\in\mathbb{R}^{\widehat{r}_n \times \widehat{r}_n}$ are orthonormal, and
		$\Sigma$$\in\mathbb{R}^{\widehat{r}_n \times \widehat{r}_n}$ is diagonal.
		\State Compute $W_n=Q U\in\mathbb{R}^{\prod_{k\neq n}I_k \times \widehat{r}_n}$
		\For{$j=1,2,\ldots,\widehat{r}_n$}
		\State $\mathbf{s}(j)=\operatorname{argmax}_{1 \leq i \leq \prod_{k\neq n}I_k}\left|(W_{n}(i, j))\right|$.
		\State $W_{n}(:, j+1)=W_n(:, j+1)-W_n(:, 1: j) \cdot(W_n(\mathbf{s}, 1: j) \backslash W_n(\mathbf{s}, j+1))$.
		\EndFor
		\State Compute $\ell_{i}=\left\|W_n(i,:)\right\| \quad$ for $i=1,2, \ldots, I_n$.
		\State Sort $\ell$ in non-increasing order.
		\State Remove entries in $\ell$ corresponding to the indices in $\mathbf{s}$.
		\State $\mathbf{s}^{\prime}=r_n-\widehat{r}_n$ indices corresponding to $r_n-\widehat{r}_n$ largest entries of $\ell$.
		\State $\mathbf{s}=\left[\mathbf{s} ; \mathbf{s}^{\prime}\right]$.
		\State Form $C_n=X_{(n)}(:,\mathbf{s})$.
		\EndFor
		\State Compute core tensor
		$\mathcal{G}\in\mathbb{R}^{r_1\times\cdots \times r_d}$ as
	%	\begin{equation*}
		$	\mathcal{G}=\mathcal{X}\times_1 C^{\dagger}_1\times_2 C^{\dagger}_2\cdots \times_d C^{\dagger}_d$.
	%	\end{equation*}
		\State {\bf return} Tucker decomposition
		$\mathcal{X}\approx\mathcal{G}\times_1 C_1   \times_2 C_2  \cdots\times_d \mathbf{C}_{d}$.
	\end{algorithmic}
\end{algorithm}
%=======Err bound of Randomized L-DEIM based HOID.
\begin{theorem}
	Let $\mathcal{X}\in\mathbb{R}^{I_1\times I_2\times \cdots \times I_d}$
	with $I_n\le\prod_{k\neq n}I_k$ for $n=1,2,\ldots,d$.
	Suppose that $p$ is an oversampling parameter, $\beta$ and $\gamma$ are positive numbers such that $\gamma>1$, and $\phi=\prod_{n=1}^d{\chi_n}$ with
		\begin{equation*}
			\chi_n=1-\frac{1}{\sqrt{2\pi(p+1)}}\left(\frac{e}{(p+1)\beta}\right)^{p+1}
			-\frac{1}{2(\gamma^2-1)\sqrt{\pi I_n\gamma^2}}\left(\frac{2\gamma^2}{e^{\gamma^2-1}}\right)^{I_n}.
		\end{equation*}
	Then Algorithm \ref{Al-R-LDEIM-HOID} provides a multilinear rank $(r_1,r_2,\ldots,r_d)$ approximation for tensor
	$\mathcal{X}$ with the following error bound which holds with probability not less than $\phi$
\begin{equation*}
	\begin{aligned}
		\left\|\mathcal{X}-\mathcal{G} \times_{1} C_{1} \times_{2} C_{2}\cdots \times_{d} C_{d}\right\|_{F}^{2} \le&
		\sum\limits_{n=1}^d
		\Bigg[
		I_n
		\left(\prod_{k\neq n}I_k\right)
		\left(
		\frac{\widehat{r}_n \cdot 4^{\widehat{r}_n}}{3}
		\right)	
		\left(
		2\sqrt{2(\widehat{r}_n+p)I_n\beta^2\gamma^2+1}
		\right.\\
		&\left.
		+2\sqrt{2(\widehat{r}_n+p)I_n}\beta\gamma
		\right)^2	
		\Bigg]
		\sigma_{\widehat{r}_n+1}^2.
	\end{aligned}
\end{equation*}
	
\end{theorem}
\begin{proof}
	By Lemma \ref{convert tensor norm to matrix}, the error in  $\mathcal{X}$ is bounded by the sum of the error in each mode, i.e.,
	\begin{equation}\label{Proof for R-L-DEIM HOID eq2}		
				\|\mathcal{X}-\mathcal{G} \times_{1} C_{1} \times_{2} C_{2}\cdots \times_{d} C_{d}\|_F^2
				\le
				\sum_{n=1}^d \left( I_n \|(I-C_nC_n^{\dagger})X_{(n)}\|_2^2\right).
	\end{equation}
	Applying the results of Theorem \ref{Err of R-LDEIM CUR}, we have
\begin{equation}\label{Proof for R-L-DEIM HOID eq1}
	\begin{aligned}
		\|(I-C_nC_n^{\dagger})X_{(n)}\|^2
		\le&
		\left(\prod_{k\neq n}I_k\right)
		\left(\frac{\widehat{r}_n 4^{\widehat{r}_n}}{3}\right)
		\left(
		2\sqrt{2\left(\widehat{r}_n+p\right)I_n\beta^2\gamma^2+1}
		\right.\\
		&\left.
		+2\sqrt{2\left(\widehat{r}_n+p\right)I_n}\beta\gamma
		\right)^2
		\sigma_{\widehat{r}_n+1}^2,
	\end{aligned}
\end{equation}
with probability not less than
%\begin{equation*}
$	\chi_n=1-\frac{1}{\sqrt{2\pi(p+1)}}\left(\frac{e}{(p+1)\beta}\right)^{p+1}
	-\frac{1}{2(\gamma^2-1)\sqrt{\pi I_n\gamma^2}}\left(\frac{2\gamma^2}{e^{\gamma^2-1}}\right)^{I_n}$.
%\end{equation*}
Setting $\phi=\prod_{n=1}^d{\chi_n}$ and plugging inequality (\ref{Proof for R-L-DEIM HOID eq1}) into (\ref{Proof for R-L-DEIM HOID eq2}), we obtain the desired result.
\end{proof}
%=================================================================================
	\section{Randomization for hybrid decomposition}
	\hskip 2em
	This subsection develops the randomized algorithms for computing the hybrid CUR-type decomposition of the form (\ref{hybrid decomposition}).
	
	\hskip 2em
	The essence of the hybrid decomposition is that we retain the fibers of the original tensor in only one mode, or in more, but not all modes.
	Specifically, as in (\ref{hybrid decomposition}), the fibers from the first $t$ modes are preserved in matrices $\{C_i\}_{i=1}^t$, which are the representative of the mode-$i$ unfolding matrices $X_{(i)}$, and matrices $\{U_j\}_{j=t+1}^d$ which contain first $r_j$ left singular vectors of $X_{(j)}$ are chosen to minimize the approximation error.
	We summarize the hybrid approach in Algorithm \ref{Al-hybrid} for the case that only the first mode of the original fibers is preserved.
%	We can also choose to extract fibers from more than one mode of $\mathcal{X}$ and the error also increases as the number of modes in which the original fibers are preserved decreases.
	Alternatively, we may opt to extract fibers from multiple modes, while noting that the reduction in the number of preserved original fibers correlates with an increase in the resulting error.
	
		\begin{algorithm}[htb]
		\caption{Hybrid algorithm \cite{begovic2022hybridBIT}}
		\label{Al-hybrid}
		\hspace*{0.02in} {\bf Require:}
		$\mathcal{X}\in\mathbb{R}^{I_1\times I_2\times\cdots \times I_d}$ and desired multilinear rank $(r_1,r_2,\ldots,r_d)$.
		\begin{algorithmic}[1]
			\For{$i=2,3,\ldots,d$}
			\State Compute matrix $U_i$ containing the leading $r_i$ left singular vectors of $X_{(i)}$.
			\EndFor
			\State Perform the PQR decomposition $X_{(1)}P=QR$.
			\State 	Compute factor matrix
			$C=X_{(1)}P(:,1:r_1)\in \mathbb{R}^{n_1\times r_1}$.
			\State Compute core tensor
			$\mathcal{G}\in\mathbb{R}^{r_1\times\cdots \times r_d}$ as
			\begin{equation*}
				\mathcal{G}=\mathcal{X}\times_1C^{\dagger}\times_2 U_2^{\mathrm{T}}\cdots \times_d U_d^{\mathrm{T}}.
			\end{equation*}
		\State {\bf return} tensor hybrid decomposition $\mathcal{X}\approx\mathcal{G}\times_1 C\times_2 U_2\cdots\times_d U_d$
		\end{algorithmic}
	\end{algorithm}
	
	\hskip 2em
	In Algorithm \ref{Al-hybrid}, the factor matrix $C$ is derived by performing PQR to the mode-$1$ unfolding, while other sampling techniques, such as the RRQR, DEIM and L-DEIM can also be employed.
	Nevertheless, the precise computation of the PQR or the singular matrices of $X_{(n)}$ can be excessively costly, %prohibitively expensive,
	thereby posing a challenge for large-scale applications. % making it difficult for large-scale applications.
	Here we adopt random sampling techniques to tackle this difficulty.
	Given matrix $X\in\mathbb{R}^{m\times n}$ with $n\ge m$, the randomized algorithm in \cite{martinsson2011ACHArandomized} yields an approximate interpolatory decomposition with the error bound
	\begin{equation}\label{error of the A=BP}
		\|CU-X\|_2\le(
		\sqrt{2lm\beta^2\gamma^2+1}(\sqrt{4k(n-k)+1}+1)
		+\beta\gamma\sqrt{2lm}\sqrt{4k(n-k)+1}
		)\sigma_{k+1},
	\end{equation}
	with probability not less than $\chi$, where $\chi$, $\beta$ and $\gamma$ are defined as in (\ref{Err bound for the randSVD}) and (\ref{Pro for the RandSVD}).
	Computational complexity analysis and numerical examples illustrate this method can accelerate the approximation of matrices significantly.
	We present our approach in Algorithm \ref{Al-R-hybrid-PQR}, where we exploit the randomization techniques to accelerate the process of the SVD and the interpolatory decomposition to each mode unfolding, providing an approximation $\widehat{\mathcal{X}}$ in a hybrid CUR-type Tucker format for a given tensor $\mathcal{X}$ such that
	\begin{equation*}
		\widehat{\mathcal{X}}
		=\mathcal{G}\times_1
		C_1\cdots \times_t C_t
		\times_{t+1}U_{t+1}\cdots\times_d U_d.
	\end{equation*}
	The following theorem quantifies the error of the approximate hybrid CUR-type Tucker decomposition produced by Algorithm \ref{Al-R-hybrid-PQR}.
	%==============Algorithm : Ramdomized hybrid algorithm
	\begin{algorithm}[htb]
	\caption{Randomized hybrid algorithm based the PQR}
	\label{Al-R-hybrid-PQR}
	\hspace*{0.02in} {\bf Require:}
	$\mathcal{X}\in\mathbb{R}^{I_1\times I_2\times\cdots \times I_d}$, multilinear rank $(r_1,r_2,\ldots,r_d)$ and oversampling parameter $p$.
	\begin{algorithmic}[1]
		\For{$i=1,2,\ldots,t$}
		\State{Draw random Gaussian matrix $\Omega\in\mathbb{R}^{(r_i+p)\times I_i}$.}
		\State{Compute $Y=\Omega X_{(i)}$.}
		\State Apply the pivoted Gram-Schmidt process to the columns of $Y$, %leaving the factorization
		%\begin{equation*}
		$	YP=QR$,
		%\end{equation*}
	%	\hspace*{0.25in}
	where $P\in \mathbb{R}^{(\prod_{k\neq i}n_k) \times (\prod_{k\neq i}n_k)}$
		$P$ is a permutation matrix,
		$Q\in \mathbb{R}^{(r_i+p) \times r_i}$ is orthonormal, and
		$R\in \mathbb{R}^{r_i \times \prod_{k\neq i}I_k}$ is
%		\hspace*{0.25in}
upper triangular.
		\State{Form $C_i=X_{(i)}P(:,1:r_i)$.}
		\EndFor
		\For{$j=t+1, t+2, \ldots,d$}
		\State{Draw random Gaussian matrix $\Omega\in\mathbb{R}^{(r_j+p)\times I_j}$.}
		\State{Compute $M=\Omega X_{(j)}$.}
		\State Compute the SVD of $M$,
		%\begin{equation*}
		$	M=Z \Sigma V^{\mathrm{T}}$,
		%\end{equation*}
	%	\hspace*{0.2in}
	where $Z\in\mathbb{R}^{\prod_{k\neq j}I_k \times (r_j+p)}$ and
		$V\in\mathbb{R}^{(r_j+p) \times (r_j+p)}$ are orthonormal,
		and $\Sigma\in\mathbb{R}^{(r_j+p)\times (r_j+p)}$
	%	\hspace*{0.25in}
		is diagonal.
		\State{Form $Q=Z(:,1:r_j)$.}
		\State{Compute $T=X_{(j)}Q$.}
		\State{Compute matrix $U_j$ containing $r_j$ left singular vectors of $X_{(j)}$.}
		\EndFor
		\State{Compute the core tensor
			$\mathcal{G}\in\mathbb{R}^{r_1\times r_2\times \cdots\times r_d}$ as
			%\begin{equation*}
			$	\mathcal{G}=\mathcal{X}
				\times_1C_1^{\dagger}\cdots \times_t C_t^{\dagger}
				\times_{t+1}U_{t+1}^{\mathrm{T}}\cdots\times_{d} U_d^{\mathrm{T}}$.
			%\end{equation*}
		}
	\end{algorithmic}
\end{algorithm}
%	\hskip 2em
	
\begin{theorem}\label{Err of the R-hybrid PQR}
	Let $\mathcal{X}\in\mathbb{R}^{I_1\times I_2\times \cdots \times I_d}$ with $I_n\le\prod_{k\neq n}I_k$ for $1\le n\le d$.
	Suppose that $p$ is an oversampling parameter, $\beta$ and $\gamma$ are positive numbers such that $\gamma>1$, and
	\begin{equation}\label{Pro of R-PQR-Hybrid}
		\chi_n=1-\frac{1}{\sqrt{2\pi(p+1)}}\left(\frac{e}{(p+1)\beta}\right)^{p+1}
		-\frac{1}{2(\gamma^2-1)\sqrt{\pi I_n \gamma^2}}\left(\frac{2\gamma^2}{e^{\gamma^2-1}}\right)^{I_n}.
	\end{equation}\label{ineq of the R-hybrid PQR}
	Then Algorithm \ref{Al-R-hybrid-PQR} produces a hybrid decomposition with the following error bound,
	\begin{equation}
		\begin{aligned}
			\|\mathcal{X}-\widehat{\mathcal{X}}\|_F^2\le
			&\sum\limits_{i=1}^t I_i
			\left[\sqrt{2(r_i+p)I_i\beta^2\gamma^2+1}
			\left(\sqrt{4r_i(\prod_{k\neq i}I_k-r_i)+1}+1\right)
			\right.\\
			&\left.
			+
			\beta\gamma\sqrt{2(r_i+p)I_i}\sqrt{4r_i(\prod_{k\neq i}I_k-r_i)+1}
			\right]^2\sigma_{r_i+1}^2
			\\
			&+ \sum\limits_{j=t+1}^{d} I_j
			\left(
			2\sqrt{2(r_j+p)I_j\beta^2\gamma^2+1}+2\sqrt{2(r_j+p)I_j}\beta\gamma
			\right)^2\sigma_{r_j+1}^2
		\end{aligned}
	\end{equation}
	with probability not less than $\phi=\prod_{n=1}^d{\chi_n}$, where $\sigma_{r_i+1}$ is the $(r_i+1)$th largest singular value of $X_{(i)}$.
\end{theorem}
\begin{proof}
	Using the property of the mode-$n$ product, we have
	\begin{equation*}
		\|\mathcal{X}-\hat{\mathcal{X}}\|_F^2
		=\|\mathcal{X}-
		\mathcal{X}\times_1(C_1C_1^{\dagger})
		\times_2\cdots\times_t(C_tC_t^{\dagger})
		\times_{t+1}(U_{t+1}U_{t+1}^{\mathrm{T}})
		\times\cdots\times_d(U_dU_d^{\mathrm{T}})\|_F^2.
	\end{equation*}
	Notice that $C_iC_i^{\dagger}$ and $U_jU_j^{\mathrm{T}}$ are orthogonal projections.
	Recalling the result in \cite[Lemma 2.1]{sorensen2016deimSIAM}, for orthogonal projections $\{\Pi_i\}_{i=1}^n$, we have
	\begin{equation*}
		\| \mathcal{X}-\mathcal{X}\times_1\Pi_1 \times_2\Pi_2  \cdots \times_d\Pi_d  \|_F^2
		\le \sum_{i=1}^d \|\mathcal{X}-\mathcal{X}\times_i\Pi_i\|_F^2,
	\end{equation*}
%	for orthogonal projections $\Pi_1,\Pi_2, \ldots,\Pi_d$,
	then it follows that
	\begin{equation}\label{the relation for Frobeniun norm of the error}
		\begin{aligned}
			\|\mathcal{X}-\widehat{\mathcal{X}}\|_F^2
			\le&\sum_{i=1}^{t}\|\mathcal{X}-\mathcal{X}\times_i(C_iC_i^{\dagger})\|_F^2
			+\sum_{j=t+1}^{d}\|\mathcal{X}-\mathcal{X}\times_j(U_jU_j^{\mathrm{T}})\|_F^2
			\\
			=&\sum_{i=1}^{t}\|(I-C_iC_i^{\dagger})X_{(i)}\|_F^2
			+\sum_{j=t+1}^{d}\|(I-U_jU_j^{\mathrm{T}})X_{(j)}\|_F^2.
		\end{aligned}
	\end{equation}
	As described in Algorithm \ref{Al-R-hybrid-PQR}, matrices $X_{(i)}, 1\le i\le t$ and $X_{(j)}, t+1\le j\le d$ own the interpolatory factorization $X_{(i)}=C_iB_i+E_i$ and the SVD such that $X_{(j)}=U_j\Sigma_jV_j^\mathrm{T}+H_j$.
	Therefore,
	\begin{equation*}
		\|(I-C_iC_i^{\dagger})X_{(i)}\|_F^2
		=\|(I-C_iC_i^{\dagger})(C_iB_i+E_i)\|_F^2
		=\|(I-C_iC_i^{\dagger})E_i\|_F^2
		\le
		I_i\|E_i\|^2_2,
	\end{equation*}
	\begin{equation*}
		\begin{aligned}
			\|(I-U_jU_j^{\mathrm{T}})X_{(j)}\|_F^2
			&=\|(I-U_jU_j^{\mathrm{T}})(U_j\Sigma_jV_j^{\mathrm{T}}+H_j)\|_F^2
			=\|(I-U_jU_j^{\mathrm{T}})H_j\|_F^2
			\le
			I_j\|H_j\|_2^2.
		\end{aligned}
	\end{equation*}
	Plugging these two inequalities into (\ref{the relation for Frobeniun norm of the error}), and using the results in (\ref{error of the A=BP}) and (\ref{Err bound for the randSVD}), we obtain the desired result.	
\end{proof}

	\hskip 2em
	Here we also apply the L-DEIM approach for designing a new randomized algorithm to compute the hybrid decomposition. %where we apply the L-DEIM algorithms to determine the indices to construct the factor matrices ${\{C_i\}}_{i=1}^t$.
	Once again, the randomized SVD algorithm is utilized to speed up the calculations.
	This method is presented in Algorithm \ref{Al-R-hybrid-LDEIM}.
	Notice that in Algorithm \ref{Al-R-hybrid-LDEIM}, if we set the parameters $\widehat{r}_i=r_i$, for $i=1,\ldots,t$, then this method degenerates to the DEIM induced HOID algorithm.
	We derive the upper bound for the expected error in the following theorem.

%======ERR for the Randomzied hybrid based on the L-DEIM
	\begin{theorem} \label{Err of the R-hybrid L-DEIM}
		Let $\widehat{\mathcal{X}}$ be an approximation of $\mathcal{X}\in\mathbb{R}^{I_1\times  I_2\times \cdots \times I_d}$ computed by Algorithm \ref{Al-R-hybrid-LDEIM}. Then the approximation error $\mathcal{E}$ satisfies
		\begin{equation}\label{ineq of the R-hybrid L-DEIM}
			\begin{aligned}
				\left\|\mathcal{E}\right\|_F^2
				&\le
				\sum\limits_{j=t+1}^d I_j
%				\left(\prod_{k\neq j}I_k\right)
				\left(
				2 \sqrt{2(r_j+p) I_j \beta^{2} \gamma^{2}+1}+2 \sqrt{2(r_j+p)I_j} \beta \gamma\right)^2
				\sigma_{r_j+1}^2
				\\
				&+
				\sum\limits_{i=1}^t I_i
				\left(\prod_{k\neq i}I_k\right)
				\left(\frac{\widehat{r}_i \cdot 4^{\widehat{r}_i}}{3}\right)
				\left(
				2\sqrt{2\left(\widehat{r}_i+p\right)I_i\beta^2\gamma^2+1}
				%				\right.\\
				%				&\left.
				+2\sqrt{2\left(\widehat{r}_i+p\right)I_i}\beta\gamma
				\right)^2
				\sigma_{\widehat{r}_i+1}^2,
			\end{aligned}
		\end{equation}
	with probability not less than $\phi=\prod_{n=1}^d{\chi_n}$. %as defined in (\ref{Pro of R-PQR-Hybrid}).
	\end{theorem}
	
	\begin{proof}
		Firstly, using analogous operation in the proof of Theorem \ref{Err of the R-hybrid PQR} yields the following inequality:
		\begin{equation}\label{Proof for R-hybrid LDEIM eq1}
			\|\mathcal{X}-\widehat{\mathcal{X}}\|_F^2
			\le
			\sum_{i=1}^{t}\|(I-C_iC_i^{\dagger})X_{(i)}\|_F^2
			+\sum_{j=t+1}^{d}\|(I-U_jU_j^{\mathrm{T}})X_{(j)}\|_F^2,
		\end{equation}
	and it still holds that
	\begin{equation}\label{Proof for R-hybrid LDEIM eq2}
		\|(I-U_jU_j^{\mathrm{T}})X_{(j)}\|_F^2
	 	\le
	 	I_j
	% 	\left(\prod_{k\neq i}I_k\right)
	 	\left(
	 	2 \sqrt{2(r_j+p) I_j \beta^{2} \gamma^{2}+1}+2 \sqrt{2(r_j+p)I_j} \beta \gamma\right)^2
	 	\sigma_{r_j+1}^2
	\end{equation}
	with probability not less than $\chi_j$.
	
	\hskip 2em
	Secondly, we note that the factor matrices $C_i$ are computed using the same method as in Algorithm \ref{Al-R-hybrid-PQR}, which generates the error satisfying that
	\begin{equation} \label{Proof for R-hybrid LDEIM eq3}
		\begin{aligned}
			\|(I-C_iC_i^{\dagger})X_{(i)}\|^2_F
			\le&
			I_i
			\left(\prod_{k\neq i}I_k\right)
			\left(\frac{\widehat{r}_i 4^{\widehat{r}_i}}{3}\right)
			\left(
			2\sqrt{2\left(\widehat{r}_i+p\right)I_i\beta^2\gamma^2+1}
			\right.\\
			&\left.
			+2\sqrt{2\left(\widehat{r}_i+p\right)I_i}\beta\gamma
			\right)^2
			\sigma_{\widehat{r}_i+1}^2,
		\end{aligned}
	\end{equation}
	with probability not less than $\chi_i$.
	Substituting (\ref{Proof for R-hybrid LDEIM eq2}) and (\ref{Proof for R-hybrid LDEIM eq3}) into the right-hand side of (\ref{Proof for R-hybrid LDEIM eq1}), the desired error bound follows.
	\end{proof}

%==============Algorithm : Ramdomized hybrid algorithm based the L-DEIM
\begin{algorithm}[t!]
	\caption{Randomized hybrid algorithm based the L-DEIM}
	\label{Al-R-hybrid-LDEIM}
	\hspace*{0.02in} {\bf Require:}
	$\mathcal{X}\in\mathbb{R}^{I_1\times I_2\times \cdots \times I_d}$,
	multilinear rank $(r_1,r_2,\ldots,r_d)$, parameters $(\widehat{r_1},\ldots,\widehat{r_t})$
	and oversampling parameter $p$.
	\begin{algorithmic}[1]
		\For{$i=1,2,\ldots,t$}
		\State Draw random Gaussian matrix
		$\Omega\in\mathbb{R}^{(\widehat{r}_i+p)\times I_i}$.
		\State Compute $Y=\Omega X_{(i)}$.
		\State Compute the SVD of $Y^{\mathrm{T}}$,
		%			\begin{equation*}
		$Y^{\mathrm{T}}=Z M K^{\mathrm{T}}$,
		%			\end{equation*}
		%			\hspace*{0.20in}
		where $Z \in\mathbb{R}^{\prod_{k\neq i}I_k \times (\widehat{r}_i+p)}$ and
		$K\in\mathbb{R}^{(\widehat{r}_i+p) \times (\widehat{r}_i+p)}$ are orthonormal,
		and $M$ is diagonal.
		\State Form $Q=Z(:,1:\widehat{r_i})$.
		\State Compute $T=X_{(i)}Q$.
		\State Compute the the SVD of $T$,
		%	\begin{equation*}
		$	T=V \Sigma G^{\mathrm{T}}$,
		%	\end{equation*}
		%			\hspace*{0.20in}
		where $V \in\mathbb{R}^{\prod_{k\neq i}I_k \times (\widehat{r}_i+p)}$ and
		$G\in\mathbb{R}^{(\widehat{r}_i+p) \times (\widehat{r}_i+p)}$ are orthonormal,
		and $\Sigma$ is diagonal.
		\State Compute $W_i=QG$.
		\For{$l=1,2,\ldots,\widehat{r}_i$}
		\State $\mathbf{s}(l)=\operatorname{argmax}_{1 \leq u \leq \prod_{k\neq i}I_k}\left|(W_i(u, l))\right|$.
		\State $W_i(:, l+1)=W_i(:, l+1)-W_i(:, 1:l) \cdot(W_i(\mathbf{s}, 1:l) \backslash W_i(\mathbf{s},l+1))$.
		\EndFor
		\State Compute $\ell_{u}=\left\|W_u(i,:)\right\| \quad$ for $u=1,2, \ldots, \prod_{k\neq i}I_k$.
		\State Sort $\ell$ in non-increasing order.
		\State Remove entries in $\ell$ corresponding to the indices in $\mathbf{s}$.
		\State $\mathbf{s}^{\prime}=r_i-\widehat{r}_i$ indices corresponding to $r_i-\widehat{r}_i$ largest entries of $\ell$.
		\State $\mathbf{s}=\left[\mathbf{s} ; \mathbf{s}^{\prime}\right]$.
		\State $C_{i}=X_{(i)}(:,\mathbf{s})$
		\EndFor
		\For{$j=t+1, t+2,\ldots,d$}
		%			\State{Draw random Gaussian matrix $\Omega\in\mathbb{R}^{(r_i+p)\times n_i}$.}
		%			\State{Compute $M=\Omega A_{(j)}$.}
		%			\State Compute the SVD of $M$,
		%			\begin{equation*}
		%				M=Z \Sigma V^{\mathrm{T}},
		%			\end{equation*}
		%			\hspace*{0.2in} where $Z\in\mathbb{R}^{\prod_{k\neq j}n_k \times (r_j+p)}$ and
		%			$V\in\mathbb{R}^{(r_j+p) \times (r_j+p)}$ are orthonormal,
		%			and $\Sigma\in\mathbb{R}^{(r_j+p)\times (r_j+p)}$ \hspace*{0.25in}is diagonal.
		%			\State{Form $Q=Z(:,1:r_j)$.}
		%			\State{Compute $T=X_{(j)}Q$.}
		\State{ Perform line 2-7 to compute matrix $U_j$ containing $r_j$ left singular vectors of $X_{(j)}$.}
		\EndFor
		\State{Compute the core tensor
			$\mathcal{G}\in\mathbb{R}^{r_1\times r_2 \times \cdots\times r_d}$ as
			%\begin{equation*}
			$	\mathcal{G}=\mathcal{X}
			\times_1C_1^{\dagger}\cdots \times_t C_t^{\dagger}
			\times_{t+1}U_{t+1}^{\mathrm{T}}\cdots\times_{d} U_d^{\mathrm{T}}$.
			%\end{equation*}
		}
	\end{algorithmic}
\end{algorithm}

	\hskip 2em
	We now comment on the practical aspects of the algorithms.
	One can observe that the randomized hybrid algorithms described above can be split naturally into two computational steps.
	The first step is to construct matrices $C_i$ for the first $t$ modes, where the arithmetic cost comprises the cost of sampling and the cost of computing the SVD or PQR and the latter is actually the most time-consuming operation.
	The computational complexity of computing SVD in lines 4 and 7 of Algorithm \ref{Al-R-hybrid-LDEIM} is
	$\mathcal{O}\left(\prod_{i=1}^t\left[(\widehat{r}_i+p)^2\prod_{k\neq i}I_k+I_i{\widehat{r}_i}^2\right]\right)$, which is much lower to the cost of computing the QR corresponding to line 4 of Algorithm \ref{Al-R-hybrid-PQR}.
%	Nevertheless, it reduces to
%	$\mathcal{O}\left(\prod_{i=1}^t\left[(\widehat{r}_i+p)^2\prod_{k\neq i}I_k+I_i{\widehat{r}_i}^2\right]\right)$
%	in Algorithm \ref{Al-R-hybrid-LDEIM},
	This computational advantage is mainly attributed to the superiority of the L-DEIM procedure, which is crucial, especially for the situation with a large $t$.
	In the second step of Algorithms \ref{Al-R-hybrid-PQR}, and \ref{Al-R-hybrid-LDEIM}, we exploit the random sampling techniques to obtain $U_j$, containing $r_j$ right singular vectors which cost
	$\mathcal{O}\left(
	\prod_{j=t+1}^d\left[
	(r_j+p)I_j\prod_{k\neq j}I_k+(r_j+p)^2\prod_{k\neq j}I_k+I_j{r_j}^2+I_jr_j\prod_{k\neq j}I_k
	\right]
	\right)$.
	Numerical experiments in the next section will show that the two algorithms lead to dramatic accelerations in practice, and have the accuracy comparable with the deterministic algorithm.

	\section{Numerical examples}
	\hskip 2em
	In this section, we check the accuracy and the computational cost of the proposed algorithms on various synthetic and real-world data sets.
	All computations are carried out in MATLAB R2020a on a computer with an AMD Ryzen 5 processor and 16 GB RAM.
	The tensor package in MATLAB, namely Tensor Toolbox \cite{bader2008SIAMefficient} is used.
%	To facilitate the comparison between different algorithms, we define the following acronyms.
	For the sake of clarity and consistency, we introduce the following acronyms to facilitate comparisons between different algorithms.
	The algorithms under consideration operate on the input tensors to produce a Tucker approximation with a multilinear rank $(r_1,r_2,\ldots,r_d)$:
%	The following algorithms act on the tensors to produce an approximation in the Tucker format with multilinear rank $(r_1,r_2,\ldots,r_d)$:
	
%	\begin{itemize}
%		\item [(1)]
	\hskip 2em 1. HOID$-$
	implements the HOID algorithm with column subset selection implemented using either the DEIM algorithm (Algorithm \ref{Al-DEIM}) labeled ``HOID-DEIM'', or the L-DEIM algorithm (Algorithm \ref{Al-LDEIM}) labeled ``HOID-LDEI'' as summarized in Algorithm \ref{Al-LDEIM-HOID}.
	
	\hskip 2em 2. R-HOID $-$
%\item [(2)]
	applies the randomized HOID algorithm with column subset selection implemented using either the DEIM algorithm (Algorithm \ref{Al-DEIM}) labeled ``R-HOID-DEIM'', summarized in Algorithm \ref{Al-R-DEIM-HOID}, or the L-DEIM algorithm (Algorithm \ref{Al-LDEIM}) labeled ``R-HOID-LDEIM'' as summarized in Algorithm \ref{Al-R-LDEIM-HOID}.
	
	\hskip 2em 3. Hybrid $-$
%\item [(3)]
	 implements Algorithm \ref{Al-hybrid} to produce the hybrid decomposition.
	
	\hskip 2em 4. R-Hybrid $-$
%\item [(4)]
	implements the randomized hybrid algorithm based on the PQR (Algorithm \ref{Al-R-hybrid-PQR}) labeled ``R-hybrid-PQR'', and the randomized algorithm based on the DEIM algorithm labeled ``R-hybrid-DEIM'', and the L-DEIM algorithm  labeled ``R-hybrid-LDEIM'' (Algorithm \ref{Al-R-hybrid-LDEIM}) to produce the hybrid decomposition.
%\end{itemize}
	
	$\mathbf{Example}$ $\mathbf{6.1}$
	  We evaluate the efficacy of the proposed algorithms on the function related tensor below from \cite{begovic2022hybridBIT}
	\begin{equation*}
		\mathcal{X}(i_1, i_2,\ldots,i_d)=\frac{1}{i_1+2\cdot i_2+\cdots+d \cdot i_d},
		%\quad 1 \leq i_{1}, \ldots, i_{d} \leq N_2
		\quad 1\le i_j \le I_j \ for\ j=1,2,\ldots,d.
	\end{equation*}
%	As described in \cite{begovic2022hybridBIT}, the advantage of using this tensor is that the singular values of each mode unfolding of tensor $\mathcal{X}$ decay rapidly, which suggests that tensor $\mathcal{X}$ is quite suitable for the randomized algorithms proposed in this paper.
	As described in \cite{begovic2022hybridBIT}, the utilization of the aforementioned tensor yields an advantage in that the singular values of every mode unfolding of the tensor $\mathcal{X}$ exhibit a rapid decay.
	This characteristic indicates that tensor $\mathcal{X}$ is highly amenable to the randomized algorithms proposed in this paper.
	The HOID and Hybrid methods are relevant here because the entries of this tensor are non-negative and we would like to preserve this structure in the column matrices $\{C_i\}$.
	
	\hskip 2em
	 We conduct two sets of experiments on tensor $\mathcal{X}$.
	 Our first experiment compares the accuracy of the HOID algorithms with their randomized counterparts R-HOID.
%	 As inputs, we take $\mathcal{X}$ with $d = 3$ and $I_j=100$ for $j=1,2,\dots,d$.
	 Our inputs consisted of tensor $\mathcal{X}$ with $d = 3$ and $I_j=100$ for $j=1,2,\dots,d$.
	 For each algorithm, we use the target multirank $(r, r, r)$, where $r$ varies from
	 $1$ to $10$, and the parameter $\widehat{r}_n$ contained in the L-DEIM procedure is $r-1$ for $r\ge 2$.
	 The same oversampling parameter $p = 5$ is used in every mode.
%	 The relative error is plotted in the left part of Figure \ref{fig-ex6.1-err}.
%	 We can see that all four algorithms have similar accuracy levels, and
%	 randomized algorithms are also highly accurate.
	The left section of Figure \ref{fig-ex6.1-err} illustrates the relative error of all four algorithms, demonstrating that their approximation errors are remarkably similar.
	Notably, the randomized algorithms exhibit impressive accuracy as well.
%	 Then we perform the same set of experiments to compare the accuracy of the Hybrid and R-Hybrid algorithms, where we set $I_j=50$ for $j=1,2,\dots,d$.
	In the subsequent experiment, we compared the accuracy of the Hybrid and R-Hybrid algorithms using a set of inputs with $I_j=50$ for $j=1,2,\dots,d$,
	and the results are depicted in the right section of Figure \ref{fig-ex6.1-err}.
%	 The results are visualized in the right part of Figure \ref{fig-ex6.1-err}. % and our conclusions are similar.
%	 Again, we observe that the four algorithms have similar performances and the error computed by the R-Hybrid-LDEIM is only slightly higher compared to the other algorithms.
	Once again, we observed that all four algorithms performed similarly, and the error computed by the R-Hybrid-LDEIM algorithm was only marginally higher than that of the other algorithms.

	 \begin{figure}[htbp]
	 	\centering
	 	\subfigure{\includegraphics[width=0.35\textwidth,height=0.35\textwidth]{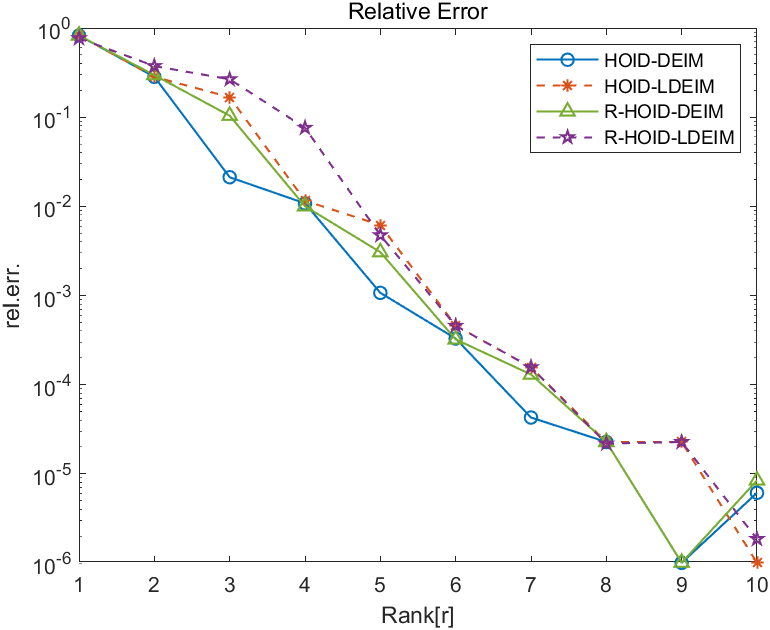}
	 	}
	 	\subfigure{\includegraphics[width=0.35\textwidth,height=0.35\textwidth]{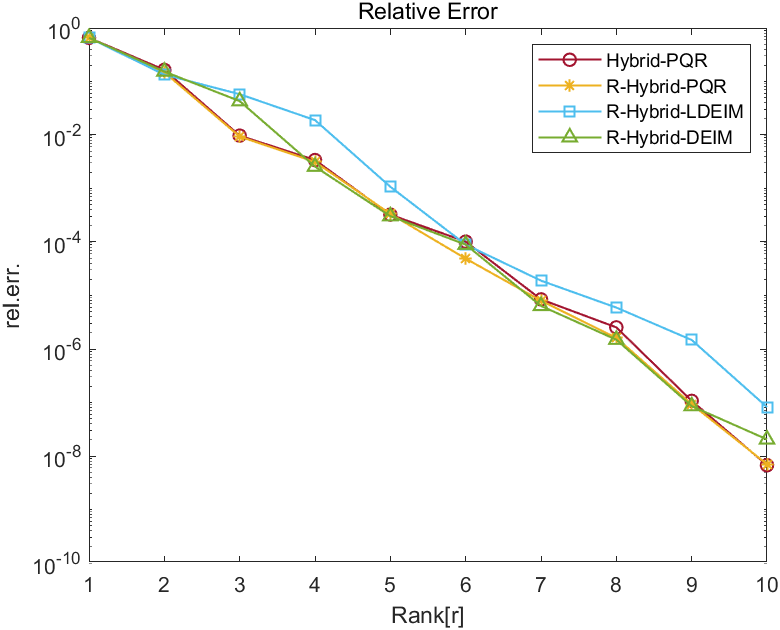}
	 	}
	 	\caption{Relative error in the computation of a rank$-(r, r, r)$ approximation to the tensor $\mathcal{X}$ with $d=3$ and oversampling parameter $p = 5$.
	 		Left: computed by the HOID and R-HOID methods with $I_j=100$ for $j=1,2,\dots,d$.
	 		Right: computed by the Hybrid and R-Hybrid methods with $I_j=50$ for $j=1,2,\dots,d$.}
 		\label{fig-ex6.1-err}
	 \end{figure}
 	\hskip 2em
 %	The analysis of the proposed algorithms implies that our randomized algorithms are less expensive compared to their deterministic counterparts.
 	Our analysis reveals that the randomized variations of our proposed algorithms exhibit significantly lower computational costs in comparison to their deterministic counterparts.
% 	To illustrate this, we record the running time in seconds (denoted as CPU) and the approximation quality (measured by the relative error, denoted as Err) of the HOID, R-HOID,
% 	Hybrid, and R-Hybrid algorithms on tensor $\mathcal{X}$ as the size of each dimension $N$ and the target rank $r$ increases.
	To illustrate this, we conducted experiments on tensor $\mathcal{X}$, gradually increasing the size of each dimension $N$ and the target rank $r$, and record the CPU time in seconds (denoted as CPU) and the approximation quality (measured by the relative error, Err) of the HOID, R-HOID, Hybrid, and R-Hybrid algorithms.
	Our investigation begin by comparing the accuracy and CPU time of the HOID algorithms against their randomized equivalents, R-HOID,
	while holding the oversampling parameter at a fixed value of $p = 5$ for the inputs.
% 	We first compare the accuracy and the CPU of the HOID algorithms with their randomized counterparts, R-HOID.
% 	For inputs, we fix the the oversampling parameter as $p = 5$.
 	According to the conclusions summarized in \cite{gidisu2022hybridArxiv}, the L-DEIM procedure may be comparable to the original DEIM method when the target rank $r$ is at most twice the available $\widehat{r}$ singular vectors.
 	Therefore, here we set the parameter $\widehat{r}_n,n=1,2,\ldots,d$ contained in the L-DEIM to be $\widehat{r}_n=r_n/2$.
 	We record the results in Table \ref{table-6.1-errCPU-HOID}.
 	It is clear from the running time that the algorithms R-HOID-DEIM and R-HOID-LDEIM have a huge advantage in computing speed over the non-random HOID method. %, and the running time of the HOID-LDEIM is slightly bigger than the HOID-DEIM.
 	%Obviously, the R-HOID-LDEIM has the fastest running time among the four algorithms.
 	 We also observe that the L-DEIM induced algorithms HOID-LDEIM and R-HOID-LDEIM beat the HOID-DEIM and R-HOID-DEIM algorithms both in terms of accuracy and computational cost.
 	
 	\hskip 2em
 	Then we perform the same set of experiments to show the advantage of the R-Hybrid over the Hybrid method, and we display the relative errors and CPU in Table \ref{table-6.1-errCPU-Hybrid}.
 	Table \ref{table-6.1-errCPU-Hybrid} illustrates that the randomized algorithms lead to a dramatic speed-up over the classical nonrandom algorithms, while the R-Hybrid-LDEIM algorithm achieves the smallest running time among the four sets of experiments.
 	We can also see that the approximation errors of all the four algorithms are very close.%similar and that the randomized algorithms are also highly accurate.
 	
 	%To see the advantage of the randomized algorithms based on the L-DEIM method over the non-random method intuitively, we also report the speed-up, which is defined as
%	\begin{equation*}
%		\text { speed-up }_{\mathrm{HOID}}
%		=\frac{\mathrm{CPU} \text { of } \text{HOID-DEIM}}
%		{\mathrm{CPU} \text { of } \text{R-HOID-LDEIM}},
%	\end{equation*}
%	\begin{equation*}
%			\text { speed-up }_{\mathrm{Hybrid}}
%		=\frac{\mathrm{CPU} \text { of } \text{Hybrid}}
%		{\mathrm{CPU} \text { of } \text{R-Hybrid-LDEIM}}.
%	\end{equation*}
 	\begin{table}[H]%[htbp]
 		\caption{Comparison of the deterministic algorithms (HOID-DEIM and HOID-LDEIM) and the randomized algorithms (R-HOID-DEIM and R-HOID-LDEIM) in the CPU and relative error as the dimension $N$ and the target rank $r$ increase.}
 		\label{table-6.1-errCPU-HOID}
 		\centering
 		\begin{tabular}{c c c c c c c}
 			\begin{tabular}{|c|c|c|c|c|c|c|}
 				\hline
 				\multicolumn{2}{|c|}{$(N,r)$ }
 				& $(200,30)$ & $(300,30)$ &$(400,40)$ & $(500,50)$ & $(600,40)$ \\
 				\hline
 				\multirow{2}{*}{ HOID-DEIM } & Err
 				& $1.5436\mathrm{e}\mbox{-}04$ & $7.7292\mathrm{e}\mbox{-}05$ & $9.5476\mathrm{e}\mbox{-}05$ & $5.9806\mathrm{e}\mbox{-}05$ & $6.7395\mathrm{e}\mbox{-}05$ \\
 				\cline { 2 - 7 } & CPU
 				& $11.018$ & $29.077$ & $89.509$ & $189.01$ & $234.53$ \\
 				\hline \multirow{2}{*}{ HOID-LDEIM } & Err
 				& $9.8135\mathrm{e}\mbox{-}07$ & $1.8748\mathrm{e}\mbox{-}08$ & $3.1492\mathrm{e}\mbox{-}05$ & $1.8460\mathrm{e}\mbox{-}05$ & $1.5446\mathrm{e}\mbox{-}05$ \\
 				\cline { 2 - 7 } & CPU
 				& $4.5589$ & $13.168$ & $37.091$ & $83.891$ & $105.76$ \\
 				\hline \multirow{2}{*}{ R-HOID-DEIM } & Err
 				& $2.7009\mathrm{e}\mbox{-}05$ & $1.2524\mathrm{e}\mbox{-}05$ & $2.7430\mathrm{e}\mbox{-}05$ & $9.2903\mathrm{e}\mbox{-}05$ & $5.0762\mathrm{e}\mbox{-}05$ \\
 				\cline { 2 - 7 } & CPU
 				& $0.69462$ & $1.7397$ & $4.9565$ & $10.936$ & $12.769$
 				\\
 				\hline \multirow{2}{*}{ R-HOID-LDEIM } & Err
 				& $1.2343\mathrm{e}\mbox{-}06$ & $6.0400\mathrm{e}\mbox{-}08$ & $5.2151\mathrm{e}\mbox{-}06$ & $3.6238\mathrm{e}\mbox{-}05$ & $3.4919\mathrm{e}\mbox{-}05$ \\
 				\cline { 2 - 7 } & CPU
 				& $0.39307$ & $1.1647$ & $2.6360$ & $5.8115$ & $7.2012$ \\
% 				\hline \multicolumn{2}{|c|}{ speed-up } & $2.45$ & $2.69$ & $2.88$ & $2.93$ & $2.76$ \\
 				\hline
 			\end{tabular}
 		\end{tabular}	
 	\end{table}
 	\begin{table}[H]%[htbp]
 	\caption{Comparison of the Hybrid and the randomized algorithms (R-Hybrid-PQR, R-Hybrid-DEIM and R-Hybrid-LDEIM) in the CPU and relative error as the dimension $N$ and the target rank $r$ increase.}
 	\label{table-6.1-errCPU-Hybrid}
 	\centering
 	\begin{tabular}{c c c c c c }
 		\begin{tabular}{|c|c|c|c|c|c|}
 			\hline
 			\multicolumn{2}{|c|}{$(N,r)$ } & $(100,20)$ &$(125,30)$&$(150,50)$ & $(200,30)$ \\
 			\hline
 			\multirow{2}{*}{ Hybrid } & Err & $2.9535\mathrm{e}\mbox{-}07$ & $1.9356\mathrm{e}\mbox{-}15$ & $1.7241\mathrm{e}\mbox{-}15$ & $2.1394\mathrm{e}\mbox{-}15$  \\
 			\cline { 2 - 6 } & CPU & $42.122$ & $135.16.3914$ & $69.017$ & $200.31$\\
 			\hline \multirow{2}{*}{ R-Hybrid-PQR } & Err & $2.6169\mathrm{e}\mbox{-}07$ & $1.5708\mathrm{e}\mbox{-}07$ & $1.7312\mathrm{e}\mbox{-}15$ & $1.7677\mathrm{e}\mbox{-}15$ \\
 			\cline { 2 - 6 } & CPU & $0.071084$ & $0.17292$ & $0.97392$ & $0.54113$ \\
 			\hline \multirow{2}{*}{ R-Hybrid-DEIM } & Err & $1.9946\mathrm{e}\mbox{-}07$ & $1.9049\mathrm{e}\mbox{-}07$ & $1.8864\mathrm{e}\mbox{-}15$ & $2.0811\mathrm{e}\mbox{-}15$\\
 			\cline { 2 - 6 } & CPU & $0.077165$ & $0.19325$ & $0.77934$ & $0.60290$
 			\\
 			\hline \multirow{2}{*}{ R-Hybrid-LDEIM } & Err & $2.1411\mathrm{e}\mbox{-}07$ & $1.6487\mathrm{e}\mbox{-}07$ & $1.9008\mathrm{e}\mbox{-}15$ & $1.5999\mathrm{e}\mbox{-}15$ \\
 			\cline { 2 - 6 } & CPU & $0.067160$ & $0.15138$ & $0.72396$ & $0.46221$ \\
% 			\hline \multicolumn{2}{|c|}{ speed-up } & $2.45$ & $2.69$ & $2.88$ & $2.93$ & $2.76$ \\
 			\hline
 		\end{tabular}
 	\end{tabular}	
 \end{table}
	
	$\mathbf{Example}$ $\mathbf{6.2}$
	Now we check the accuracy and the computational cost of our algorithms on real-world tensors.
	Our first test problem comes from the classification of handwritten digits images.
	This problem, popularized by Savas and $\text { Eldén }$ in \cite{savas2007PRhandwritten}, involves assigning a label from 0-9 to a new image representing a handwritten digit.
%	Grayscale images of handwritten digits from 0–9	are available in 10 classes, as a training data set;
%	given a new image representing a handwritten digit, the challenge is to assign it a label from 0–9.
	Here we adopt the classification strategy in \cite{saibaba2016hoidSIAM} which relies on the HOID representation and consists of two main steps:
	a compression phase and a classification phase.
	In the compression phase, various approaches are applied to a training image dataset arranged as a tensor to compute a low multirank decomposition, while the second step is a classification phase.
	Our focus in this study is on the first step of efficiently decomposing a tensor formed using images from the MNIST and USPS databases \cite{OlivettiFaces}.
	These databases contain 60,000 images with $28\times 28$ pixels and 1,100 images with $16\times 16$ pixels, both in 8-bit grayscale.
	The images are unequally distributed over ten classes, but to ensure equal representation across all digits in MNIST, we restrict the number of images in each class to 5,421.
	Consequently, we organize the images from MNIST and USPS into tensors of size $784 \times 5421 \times 10$ and $256\times 110 \times 10$, respectively.
	Here, the first dimension represents the pixels, the second dimension represents the images, and the third dimension represents the digits.
%	The first is a compression phase, where approaches are applied to a training images dataset, arranged as a tensor, to compute a low mulitirank decomposition.
%	The second step is a classification phase, in which the decomposed tensor is used to classify images in the test dataset.
%	Here we focus on the first step to efficiently decompose the tensor formed using images from the MNIST and the USPS database \cite{OlivettiFaces}, which contain 60000 images with $28 \times 28$ pixels and 1100 images with $16 \times 16$ pixel both in 8-bit grayscale.
%	The images are unequally distributed over the ten classes;
%	To keep the same number over all the digits in the MINST, we restrict the number of images in every class to 5421.
%	Thus, the images from MINST and the USPS are organized into a tensor of size $784 \times 5421 \times 10$ and $256\times 110 \times 10$ respectively, with the first dimension representing the pixels, the second dimension representing the images, and the third dimension representing the digits.
	
	\hskip 2em
	Specifically, we fix the target multirank $(r_1,r_2,r_3)$ to be (62,142,10) for the MINST, (50,100,10) for the USPS, the oversampling parameter $p = 5$ and the parameter $\widehat{r}_n$ contained in the L-DEIM as $r_n/2$ for $n=1,2,3$.
	For the Hybrid methods, the original fibers only in the first mode are preserved.
	We report the running time and the relative error of the HOID, Hybrid algorithms and its randomized algorithms in Tables \ref{table-6.2-errCPU-HOID} and \ref{table-6.2-errCPU-Hybrid}.
	We observe that the classical algorithm runs almost three times as long as the L-DEIM induced randomized algorithms
    (R-HOID-LDEIM and R-Hybrid-LDEIM), which also give  comparable relative errors.
	It indicates that using the random sampling techniques and L-DEIM method leads to a dramatic speed-up over classical techniques.
	
	\begin{table}[H]%[htbp]
		\caption{Relative error and running time of both HOID and the randomized algorithms R-HOID on the tensors generated from the USPS and MINST database.
		We set the target multirank $(r_1,r_2,r_3)$ to be $(62,142,10)$ for MINST and $(50,100,10)$ for the USPS, while the parameter $\widehat{r}_n$ contained in the L-DEIM is $r_n/2$ for $n=1,2,3$ and an oversampling parameter of $p = 5$ as inputs.}
		\label{table-6.2-errCPU-HOID}
		\centering
		\begin{tabular}{|c|c|c|c|c|c|}
			\hline
			\multicolumn{2}{|c|}{Method } &HOID-DEIM &HOID-LDEIM&R-HOID-DEIM &R-HOID-LDEIM \\
			\hline
			\multirow{2}{*}{ USPS } & Err & $0.53785$ & $0.63637$ & $0.56825$ & $0.65947$  \\
			\cline { 2 - 6 } & CPU & $4.3392$ & $2.0040$ & $0.63923$ & $0.34326$\\
			\hline \multirow{2}{*}{ MINST } & Err & $0.54002$ & $0.68095$ & $0.56040$ & $0.70049$ \\
			\cline { 2 - 6 } & CPU & $64.103$ & $32.175$ & $8.8650$ & $4.9925$ \\
			\hline
		\end{tabular}
	\end{table}
	
	\begin{table}[H]%[htbp]
		\caption{Relative error and running time of both Hybrid, R-Hybrid-PQR, R-Hybrid-DEIM and R-Hybrid-LDEIM on the tensors generated from the USPS and MINST database.
		The parameters and target mulitrank are the same as in Table \ref{table-6.2-errCPU-HOID} and the original fibers only in the first mode are preserved. }
		\label{table-6.2-errCPU-Hybrid}
		\centering
		\begin{tabular}{|c|c|c|c|c|c|}
			\hline
			\multicolumn{2}{|c|}{Method } &Hybrid &R-hybrid-PQR&R-hybrid-DEIM &R-hybrid-LDEIM \\
			\hline
			\multirow{2}{*}{ USPS } & Err & $0.47248$ & $0.52994$ & $0.52791$ & $0.54056$  \\
			\cline { 2 - 6 } & CPU & $1.9945$ & $0.42849$ & $0.49049$ & $0.39817$\\
			\hline \multirow{2}{*}{ MINST } & Err & $0.48134$ & $0.52807$ & $0.51707$ & $0.53724$ \\
			\cline { 2 - 6 } & CPU & $127.55$ & $5.9234$ & $6.5194$ & $5.6317$ \\
			\hline
		\end{tabular}
	\end{table}
	
	$\mathbf{Example}$ $\mathbf{6.3}$
%	Our final test  come from the formidable repository of sparse tensors and tools (FROSTT) database \cite{smith2017frostt}.
%	We choose two  large and sparse tensors whose features are summarized in Table \ref{ex6.3-Original information}.
%	The NELL-2 \cite{carlson2010AAAItoward} is a dataset that can be used in machine learning system that relates different entities, creating a three-dimensional dataset whose modes represent entity, relation,	and entity. The NIPS Publications dataset \cite{chechik2007JMLReec} contains the papers published in NIPS from 1987 to 2003, collected by Globerson {\it  et al.}
%	The modes represent paper-author-word-year, and the values are counts of words.
	We conducted our final test using a formidable repository of sparse tensors and associated tools, namely, the FROSTT database \cite{smith2017frostt}.
	For this purpose, we selected two large and sparse tensors, whose salient characteristics are presented in Table \ref{ex6.3-Original information}.
	The first tensor, NELL-2 [9], is a dataset that is commonly employed in machine learning systems for establishing relationships among various entities.
	It is a three-dimensional dataset, where the modes correspond to entity, relation, and entity, respectively.
	The second tensor, the NIPS Publications dataset \cite{chechik2007JMLReec}, was collected by Globerson {\it  et al.} and contains papers published in NIPS between 1987 and 2003.
	The tensor has four modes that correspond to paper, author, word, and year, respectively. The entries of the tensor denote the frequency of the occurrence of words in each paper.
{\small	
	 \begin{table}[H]%[htbp]
	 	\centering
	 \caption{
		Summary of sparse tensor examples from the FROSTT database-we include the details for both the full datasets and the condensed datasets used in our experiments.}
	\label{ex6.3-Original information}
	\begin{tabular}{c|c|c|c}
		\hline
		Original tensor & Order & Size & Nonzeros \\
		\hline
		NELL-2 & 3 & $12092 \times 9184 \times 28818$ & $76,879,419$ \\
		NIPS & 4 & $2482 \times 2862 \times 14036 \times 17$ & $3,101,609$ \\
		\hline Condensed tensor & Order & Size & Nonzeros \\
		\hline NELL-2 & 3 & $532 \times 682 \times 606$ & $7069$ \\
		NIPS & 3 & $632 \times 647 \times 684$ & $4561$ \\
		\hline
	\end{tabular}
    \end{table}
}	

	\hskip 2em
	First, we ran both the HOID and R-HOID algorithms on the NELL-2 and NIPS defined in Table \ref{ex6.3-Original information}, which produce a multirank-$(20,20,50)$ and a multirank-$(50,50,50)$ approximation respectively.
	As inputs to our test algorithms, we use the parameter of the L-DEIM $(\widehat{r}_1,\widehat{r}_2,\widehat{r}_3)=(15,15,30)$ for the NELL-2 and use $(\widehat{r}_1,\widehat{r}_2,\widehat{r}_3)=(25,25,25)$ for the NIPS, and the oversampling parameter $p=5$.
	Then we ran the Hybrid and the R-Hybrid algorithm on the NELL 2 and NIPS, where we keep the same parameters and the target mulitrank and we preserve the first two modes of the original tensors.
	The corresponding results are displayed in Tables \ref{ex6.3-ErrCPU-HOID} and \ref{ex6.3-ErrCPU-Hybrid}, where we can see that the randomized algorithms give comparable relative errors at substantially less cost.

		\begin{table}[t!]%[htbp]
		\caption{Relative error and running time of both HOID and R-HOID on the tensors defined in Table \ref{ex6.3-Original information}. We set the target multirank of $(20,20,50)$ and $(\widehat{r}_1,\widehat{r}_2,\widehat{r}_3)=(15,15,30)$ for the NELL-2, while a multirank of $(50,50,50)$ and $(\widehat{r}_1,\widehat{r}_2,\widehat{r}_3)=(25,25,25)$ for the NIPS and an oversampling parameter of $p = 5$ as inputs.}
		\label{ex6.3-ErrCPU-HOID}
		\centering
		\begin{tabular}{|c|c|c|c|c|c|}
			\hline
			\multicolumn{2}{|c|}{Method } &HOID-DEIM &HOID-LDEIM&R-HOID-DEIM &R-HOID-LDEIM \\
			\hline
			\multirow{2}{*}{ NELL-2 } & Err & $0.086911$ & $0.10288$ & $0.090825$ & $0.10776$  \\
			\cline { 2 - 6 } & CPU & $95.989$ & $59.590$ & $7.1387$ & $4.8956$\\
			\hline \multirow{2}{*}{ NIPS } & Err & $0.50393$ & $0.66685$ & $0.56733$ & $0.70049$ \\
			\cline { 2 - 6 } & CPU & $175.83$ & $96.301$ & $10.297$ & $6.1198$ \\
			\hline
		\end{tabular}
	\end{table}
	
	\begin{table}[H]%[htbp]
		\caption{Relative error and running time of Hybrid, R-Hybrid-PQR, R-Hybrid-DEIM and R-Hybrid-LDEIM on the tensors defined in Table \ref{ex6.3-Original information}.
		The parameters and target mulitrank are the same as in Table \ref{ex6.3-ErrCPU-HOID} and we preserve the first two modes of the original tensors. }
		\label{ex6.3-ErrCPU-Hybrid}
		\centering
		\begin{tabular}{|c|c|c|c|c|c|}
			\hline
			\multicolumn{2}{|c|}{Method } &Hybrid &R-hybrid-PQR&R-hybrid-DEIM &R-hybrid-LDEIM \\
			\hline
			\multirow{2}{*}{ NELL-2 } & Err & $0.086615$ & $0.090311$ & $0.088822$ & $0.11510$  \\
			\cline { 2 - 6 } & CPU & $118.59$ & $8.4060$ & $9.7856$ & $8.7567$\\
			\hline \multirow{2}{*}{ NIPS } & Err & $0.43845$ & $0.50652$ & $0.49945$ & $0.65049$ \\
			\cline { 2 - 6 } & CPU & $157.99$ & $15.505$ & $19.600$ & $12.747$ \\
			\hline
		\end{tabular}
	\end{table}
	\section{Conclusion}
	\hskip 2em
	In this paper, by combining the random sampling techniques with the L-DEIM method, we develop new efficient randomized algorithms for computing the approximate CUR-type and hybrid CUR decomposition for tensors in the Tucker format with a given target multilinear rank.
	We also provided the detailed probabilistic analysis for the proposed randomized algorithms.
	Theoretical analysis and numerical examples illustrate that exploiting the randomized
	techniques results in a big improvement in terms of the CPU time while keeping a high degree of accuracy.
	Finally, it is natural to consider applying the L-DEIM for developing randomized algorithms that adaptively find a low multirank representation satisfying a given tolerance,
	which is particularly useful when the target rank is not known in advance, and it will be discussed in our future work.

%%%%%%%%%%%%%%%%%%%%%%	
  \section*{Acknowledgments}
        This work is supported by the National Natural Science Foundation of China (No. 12271108 and 11801534), the Innovation Program of Shanghai Municipal Education Committee and the Fundamental Research Funds for the Central Universities (No. 202264006).

	\bibliographystyle{siam}
	\bibliography{ldeimHOID}	
\end{document}